\documentclass[11pt,a4paper]{amsart} 
\usepackage[czech,english]{babel}
\usepackage{amscd,amssymb}
\usepackage{amsmath}
\usepackage{mathtools}
\usepackage{mathrsfs}
\usepackage{dsfont}
\usepackage{url}
\usepackage[all]{xy}
\usepackage{hyperref}
\usepackage{tikz-cd}

\title{$S$-almost perfect commutative rings}

\thanks{
The first-named author is partially supported by grants
BIRD163492 and DOR1690814 of Padova University.}

\thanks{
The second-named author's research is supported by
research plan RVO:~67985840, the Israel
Science Foundation grant~\#\,446/15, and by the Grant Agency
of the Czech Republic under the grant P201/12/G028.}

\author[S.~Bazzoni]{Silvana Bazzoni}
\address[Silvana Bazzoni]{%
Dipartimento di Matematica Tullio Levi-Civita \\
Universit\`a di Padova \\
    Via Trieste 63, 35121 Padova (Italy)}
\email{bazzoni@math.unipd.it}

\author[L.~Positselski]{Leonid Positselski}
\address[Leonid Positselski]{%
Institute of Mathematics \\
Czech Academy of Sciences \\
\v Zitn\'a~25, 115~67 Praha~1 (Czech Republic); and
\newline\indent
Laboratory of Algebraic Geometry \\
National Research University Higher School of Economics \\
Moscow 119048 (Russia); and
\newline\indent
Sector of Algebra and Number Theory \\
Institute for Information Transmission Problems \\
Moscow 127051 (Russia); and
\newline\indent
Department of Mathematics \\
Faculty of Natural Sciences \\
University of Haifa \\
Mount Carmel, Haifa 31905 (Israel)}
\email{positselski@yandex.ru}
\renewcommand{\iff}{if and only if }

\newcommand{\la}{\longrightarrow}

\newcommand{\ep}{\varepsilon}

\newcommand{\bbN}{\mathbb{N}}

\newcommand{\m}{\mathfrak{m}}
\newcommand{\n}{\mathfrak{n}}
\newcommand{\p}{\mathfrak{p}}
\newcommand{\q}{\mathfrak{q}}

\newcommand{\Max}{\operatorname{Max}}    

\newcommand{\Hom}{\operatorname{Hom}}

\newcommand{\Ext}{\operatorname{Ext}}
\newcommand{\Tor}{\operatorname{Tor}}

\newcommand{\Gen}{\operatorname{Gen}}
\newcommand{\Ker}{\operatorname{Ker}}

\newcommand{\Img}{\operatorname{Im}}
\newcommand{\Coker}{\operatorname{Coker}}

\newcommand{\A}{\mathcal{A}}
\newcommand{\B}{\mathcal{B}}
\newcommand{\C}{\mathcal{C}}
\newcommand{\D}{\mathcal{D}}

\newcommand{\F}{\mathcal{F}}

\newcommand{\clP}{\mathcal{P}}

\newcommand{\SF}{\mathcal{S}\mathcal{F}}
\newcommand{\WC}{\mathcal{W}\mathcal{C}}

\newcommand{\W}{\mathcal{W}}
 
\newcommand{\Modr}[1]{\mathrm{Mod}\textrm{-}{#1}}
\newcommand{\Modl}[1]{{#1}\textrm{-}\mathrm{Mod}}

\newcommand{\ModR}{\mathrm{Mod}\textrm{-}R}

\newcommand{\Ann}{\mathrm{Ann}}
 
\newcommand{\Der}[1]{\mathbf{D}({#1})}

\theoremstyle{plain}
\newtheorem{thm}{Theorem}[section]
\newtheorem{lem}[thm]{Lemma}
\newtheorem{prop}[thm]{Proposition}
\newtheorem{cor}[thm]{Corollary}

\theoremstyle{definition}
\newtheorem{defn}[thm]{Definition}

\newtheorem{nota}[thm]{Notation}
\newtheorem{expl}[thm]{Example}

\theoremstyle{remark}
\newtheorem{rem}[thm]{Remark}

\newcommand{\+}{\protect\nobreakdash-}

\newcommand{\aaa}{\mathfrak a}

\begin{document}

\begin{abstract}
 Given a multiplicative subset $S$ in a commutative ring $R$, we
consider $S$\+weakly cotorsion and $S$\+strongly flat $R$\+modules,
and show that all $R$\+modules have $S$\+strongly flat covers
if and only if all flat $R$\+modules are $S$\+strongly flat.
 These equivalent conditions hold if and only if the localization $R_S$
is a perfect ring and, for every element $s\in S$, the quotient ring
$R/sR$ is a perfect ring, too.
 The multiplicative subset $S\subset R$ is allowed to contain
zero-divisors.
\end{abstract}

\maketitle

\setcounter{tocdepth}{1}
\tableofcontents

\section*{Introduction}

 Let $R$ be a commutative ring and $Q$ its total ring of quotients.
 An $R$\+module $C$ is said to be \emph{weakly cotorsion} if
$\Ext_R^1(Q,C)=0$.
 An $R$\+module $F$ is \emph{strongly flat} if $\Ext_R^1(F,C)=0$
for all weakly cotorsion $R$\+modules~$C$.
 This definition first appeared in the paper~\cite{Trl1}.
 The problem of characterizing commutative domains $R$ for which
the class of all strongly flat modules is covering was posed in
lecture notes~\cite{Trl2}.

 This problem was solved  in the series of papers~\cite{BS1,BS2}.
 It was shown that, for a commutative domain~$R$, the class of
all strongly flat $R$\+modules is covering if and only if it coincides
with the class of all flat $R$\+modules, and this holds if and only
if all the quotient rings of $R$ by nonzero ideals are perfect.
 Such rings $R$ received the name of \emph{almost perfect domains},
and were further studied in the papers~\cite{BS3,Baz}.

 Some of these results were generalized to commutative rings $R$ with
zero-divisors in the recent paper~\cite{FS}.
 There it was shown that, for a given commutative ring $R$, all flat
$R$\+modules are strongly flat if and only if the following two
conditions are satisfied.
 Firstly, the ring $Q$ must be perfect, and secondly, all the quotient
rings of $R$ by its principal ideals generated by regular elements must
be perfect as well.

 In an independent development, a partial extension of the results
of~\cite{BS1,BS2} to the following setting was suggested in
the paper~\cite{Pcta}.
 Let $R$ be a commutative ring and $S\subset R$ be its multiplicative
subset (which may well contain some zero-divisors).
 Then one can consider the localization $R_S$ of $R$ with respect to $S$,
and define \emph{$S$\+weakly cotorsion} and \emph{$S$\+strongly
flat} $R$\+modules by replacing $Q$ with $R_S$ in the above definition.

 It was shown in~\cite{Pcta} that, for any commutative Noetherian ring
$R$ of Krull dimension~$1$, denoting by $S$ the complement to
the union of all minimal prime ideals in $R$, one has that the classes of
flat and $S$\+strongly flat $R$\+modules coincide.
 The argument was based on the novel \emph{contramodule} techniques
developed in the papers~\cite{Pcta,Pos}, together with the idea of
using the two-term complex $R\to R_S$ where the quotient module
$K=Q/R$ was traditionally considered.

 The aim of this paper is to characterize multiplicative subsets $S$ in
commutative rings $R$ for which the class of all $S$\+strongly flat
$R$\+modules is covering.  Once again, it turns out that this holds if
and only if this class coincides with the class of all flat $R$\+modules.
 Commutative rings $R$ for which all flat $R$\+modules are
$S$\+strongly flat are characterized by the following two conditions:
the ring $R_S$ is perfect, and, for every element $s\in S$,
the ring $R/sR$ are perfect, too (Theorem~\ref{T:characterization}).
 Rings satisfying these conditions are called \emph{$S$\+almost perfect}.

 A general framework for $S$\+strongly flat modules was developed
in the paper~\cite{PSl2}, where it was shown that, under moderate
assumptions on a multiplicative subset $S$ in a commutative ring $R$,
a flat $R$\+module $F$ is $S$\+strongly flat if and only if
the $R_S$\+module $F\otimes_RR_S$ is projective and
the $R/sR$\+module $F/sF$ is projective for every $s\in S$.
 We show that such description of $S$\+strongly flat modules is valid
for all \emph{$S$\+h-nil} rings, which are defined as the rings $R$ such
that every element $s\in S$ is contained only in finitely many maximal
ideals of $R$ and every prime ideal of $R$ intersecting $S$ is maximal.

 By Theorem~\ref{T:S-h-nil-is-Matlis} we also show that, for any
$S$\+h-nil ring $R$, the projective dimension of the $R$\+module
$R_S$ does not exceed~$1$.
 This is an important property of a multiplicative subset in
a commutative ring, known to be equivalent to a number of other
conditions, at least, in the case when $S$ consists of
regular elements~\cite{AHHT}.
 We deduce some implications of this property without the regularity
assumption.
 We also define and discuss \emph{$S$\+h-local rings}, which form
a wider class than that of $S$\+h-nil rings.

 Completions with respect to the $R$\+topology played a key role
in the arguments in the papers~\cite{BS1,BS2}.
 \emph{Contramodules}, which appear to be a more powerful and
flexible version of $R$\+ and $S$\+completions and complete
modules, suitable for applications to homological algebra questions,
are used in this paper instead.
 We refer the reader to the papers~\cite{Pcta,Pos} for background
material about contramodules.

 Let us describe the overall structure of the paper.
 In Section~\ref{preliminaries-secn} we collect a few known or
essentially known results supported by brief arguments and
references, while Section~\ref{S-divisible-torsion-secn}
contains several new or more complicated technical lemmas with
full proofs.
 We prove our criterion for existence of $S$\+strongly flat covers
in Section~\ref{strongly-flat-secn}.

 The discussion of $S$\+h-local rings and decompositions of modules
into direct sums or products indexed over the maximal ideals of
the ring in Section~\ref{S-h-local-secn} forms a background for
the discussion of $S$\+h-nil rings in
Section~\ref{S-h-nil-secn}.
 In Section~\ref{t-contra-secn} we discuss $t$\+contramodules, which
are our main technical tool.
 In particular,
Proposition~\ref{t-contramodules-relatively-cotorsion}
is an important result, going back to~\cite{Pcta}, on which
the proofs in Section~\ref{S-almost-perfect-secn} are based.

 The main results of the paper are proved in
Sections~\ref{S-h-nil-secn}\+-\ref{P1=F1-secn}.
 These are Theorem~\ref{T:S-h-nil-is-Matlis}, telling that
the projective dimension of the $R$\+module $R_S$ does not exceed~$1$
for an $S$\+h-nil ring~$R$; Theorem~\ref{T:characterization},
characterizing rings $R$ with a multiplicative subset $S$ for which
the class of $S$\+strongly flat $R$\+modules is covering, as mentioned
above; Proposition~\ref{P:S-strongly-flat-over-S-h-nil}, characterizing
$S$\+strongly flat modules over an $S$\+h-nil ring;
Proposition~\ref{P:P_1=F_1}, claiming that all modules of flat
dimension not exceeding~$1$ over an $S$\+almost perfect ring
have projective dimension not exceeding~$1$; and
Proposition~\ref{P:regular-P_1}, proving a partial inverse implication
and providing further characterizations of $S$\+almost perfect rings
in the case when $S$ consists of regular elements.

 The authors are grateful to Luigi Salce and Jan Trlifaj for helpful
discussions.
 We also thank the referee for several suggestions which helped us
to improve the revised version of the paper.

\section{Preliminaries} \label{preliminaries-secn}

 Throughout the paper $R$ will be a commutative ring and $S$ a multiplicative subset of~$R$.
 We let $R\overset{\phi}\to R_S$ be the localization map
and put $I=\Ker\phi$.

 We denote by $\Modl R$ the category of $R$\+modules.

 An $R$\+module $M$ is said to be \emph{$S$\+torsion} if for every element
$x\in M$ there exists an element $s\in S$ such that $sx=0$.
 An $R$\+module $M$ is $S$\+torsion if and only if $R_S\otimes_RM=0$.
 The maximal $S$\+torsion submodule of an $R$\+module $M$ is denoted
by $\Gamma_S(M)\subset M$.

 An $R$\+module $D$ is \emph{$S$\+divisible} if $sD=D$ for every $s\in S$,
and it is \emph{$S$\+h-divisible} if it is an epimorphic image of $R_S^{(\alpha)}$
for some cardinal~$\alpha$, or equivalently,  if every $R$\+module morphism
$R\to D$ extends to~$R_S$.

 An $R$\+module $M$ is \emph{$S$-reduced} if it has no nonzero $S$\+divisible submodules, and it is \emph{$S$\+h-reduced} if it has no nonzero $S$\+h-divisible submodules, or equivalently, if $\Hom_R(R_S, M)=0$.

 For every $n\geq 0$, let $\clP_n(R)$ ($\F_n(R)$\,) denote the class of all
$R$\+modules of projective (flat)  dimension at most~$n$.
 We drop to mention the ring $R$ when there is no possibility of confusion.

 An $R$\+module $M$ is \emph{$S$\+weakly cotorsion} ($S$\+$\WC$) if
$\Ext^1_R(R_S, M)=0$, and it is \emph{Enochs cotorsion} ($Cot$) if 
$\Ext^1_R(F, M)=0$ for every flat module $F$.

 An $R$\+module which is $S$\+h-reduced and  $S$\+weakly cotorsion
is called an \emph{$S$\+contramodule} (see~\cite{Pos}).

\begin{lem} \label{S-contra}
\begin{enumerate}
\item Let $f\colon R\to R'$ be a homomorphism of commutative rings,
        $S\subset R$ a multiplicative subset, and $M$ an $R'$\+module.
        Then $M$ is $S$\+h-reduced ($S$\+weakly cotorsion, or
        $S$\+contramodule) as an $R$\+module if and only if it is
        $f(S)$\+h-reduced ($f(S)$\+weakly cotorsion, or
        $f(S)$\+contramodule, respectively) as an $R'$\+module.
\item Let $S\subset T\subset R$ be
        multiplicative subsets of $R$.
        Then any $S$\+contramodule $R$\+module is also
        a $T$\+contramodule $R$\+module.
\item The full subcategory of all $S$\+contramodule $R$\+modules
        is closed under the kernels of morphisms, extensions, and
        infinite products in the category $\Modl R$.
\item Assume that\/ $\operatorname{p.dim}_R R_S\leq 1$.
        Then the full subcategory of $S$\+contramodule $R$\+modules
        is also closed under cokernels in $\Modl R$.
\end{enumerate}
\end{lem}

\begin{proof}
 (1) follows from the isomorphism $\Ext_R^i(R_S,M)\cong
\Ext_{R'}^i(R'_{f(S)},M)$, which holds for all $i\ge0$.

 (2) is an easy computation using that $R_T$ is an $R_S$\+module
(see~\cite[Lemma~1.2]{Pos} for details).

 (3) is a particular case of~\cite[Proposition~1.1]{GL}.

 (4) is also a particular case of~\cite[Proposition~1.1]{GL}
or~\cite[Theorem~1.2(a)]{Pcta}; see~\cite[Theorem~3.4(a)]{Pos}.
\end{proof}

 For a class $\C$ of $R$\+modules, $\C^{\perp_1}$ denotes the class of all
$R$\+modules $M$ such that $\Ext^1_R(C, M)=0$ for every $C\in \C$,
and symmetrically  $^{\perp_1}\C$ is the class of all $R$\+modules $M$
such that $\Ext^1_R(M, C)=0$ for every $C\in \C$. 
 
 A pair $(\A, \B)$ of classes of $R$\+modules is called a cotorsion pair if
$\A^{\perp_1}=\B$ and $^{\perp_1}\B=\A$.
 For the notion of a complete cotorsion pair we refer to~\cite{GT}.
 
 An $R$\+module $F$ is \emph{$S$-strongly flat} ($S$\+$\SF$) if
$\Ext^1_R(F, C)=0$ for every $S$\+weakly cotorsion $R$\+module $C$.
 We have that $(S$\+$\SF,S$\+$\WC)$ is a complete cotorsion pair;
$S$\+$\SF\subseteq \F_0$ and $Cot\subseteq S$\+$\W\C$.

\begin{lem}\label{L:SF}
 An $R$\+module $F$ is $S$\+strongly flat \iff $F$ is a direct summand
of an $R$\+module $G$ fitting into a short exact sequence of the type
\begin{equation}
\tag{SF} 0\la R^{(\beta)}\la G\overset{\pi}\la R_S^{(\gamma)}\la 0
\end{equation}
 for some cardinals $\beta$ and~$\gamma$. 
\end{lem}

\begin{proof} Sufficiency is clear. 
 For the necessary condition, note that   $\Ext^1_R(R_S, R_S^{(\alpha)})=0$
for every cardinal $\alpha$.
 Hence, the conclusion follows by~\cite[Corollary~6.13]{GT}.
\end{proof}
 
 Recall that a  ring $R$ is \emph{right perfect} \iff every flat right
$R$\+module is projective.
 The \emph{right big finitistic dimension} $\operatorname{r.FPdim}R$
of a ring $R$ is the supremum of the projective dimensions of
the right $R$\+modules of finite projective dimension.

 An ideal $J$ of $R$ is \emph{left\/ $T$-nilpotent} if for every sequence
$a_1$, $a_2$,~\dots, $a_n$,~\dots\ of elements of $J$ there is~$m$
such that $a_1 a_2 \dotsm a_m=0$.
  
 The following characterization is well-known. 
\begin{lem}\label{L:perfect}
 Let $R$ be a commutative ring. Then the following conditions
are equivalent:
\begin{enumerate}
\item $R$ is a perfect ring.
\item The big finitistic dimension $\operatorname{FPdim}R$ is
equal to\/~$0$. 
\item $R$ is a finite product of local rings, each one with
a $T$-nilpotent maximal ideal.
\item $R$ is semilocal and semiartinian, i.e., every nonzero factor of $R$
contains a simple $R$\+module.
\end{enumerate}
Moreover, every regular element of a commutative perfect ring is invertible.
\end{lem}

\begin{proof}
 For the last statement note that if $r$~is a regular element of $R$,
then the $R$\+module $R/rR$ has projective dimension at most one,
hence it is projective.
 So $rR=eR$ for some idempotent element $e\in R$.
 Since $r$~is regular it must be $e=1$ and so $r$~is invertible.
\end{proof}
\section{$S$-divisible and $S$-torsion modules}
\label{S-divisible-torsion-secn}
%

 In this section we collect some properties of $S$-divisible,
$S$\+h-divisible, $S$\+torsion and $S$\+contramodule $R$-modules.

\begin{prop}\label{P:Matlis} The following hold true:
\begin{enumerate}
\item Every $S$\+divisible $R$\+module is annihilated by~$I$.
\item Assume that\/ $\operatorname{p.dim}_R R_S\leq 1$.
Then, $\operatorname{p.dim}_{R/I}R_S\leq 1$ and
every $S$\+divisible $R$\+module is $S$\+h-divisible.
\end{enumerate} 
\end{prop}

\begin{proof} 
 (1) Let $D$ be an $S$\+divisible $R$\+module and let $a\in I$.
 Then, there is $s\in S$ such that $as=0$.
 So given $x\in D$, we have $x=sy$ for some $y\in D$, and $ax=asy=0$.
 
(2) $R_S$ is the localization of $R/I$ at the multiplicative set
$\overline S=S+I/I$, which consists of regular elements of~$R/I$.
 Tensoring a projective presentation of the $R$\+module $R_S$ by $R/I$,
we see that $\operatorname{p.dim}_{R/I}R_S\leq 1$ (since $R_S$ is
a flat $R$\+module).
 Hence we can apply~\cite[Theorem~1.1 or Proposition~6.4]{AHHT}
to the ring $R/I$ and the regular multiplicative set $\overline S$
to conclude that the class of $R/I$-modules $\Gen R_S$ generated by
$R_S$ coincides with the class of all $R/I$-modules which are
$\overline S$\+divisible.
 This last is easily seen to coincide with the class of all $S$\+divisible
$R$-modules, since, by~(1), every $S$\+divisible $R$\+module is
annihilated by~$I$.
 Thus, the classes of the $S$\+h-divisible and $S$\+divisible
$R$\+modules coincide. 
\end{proof}
\begin{lem}\label{L:annihil}
 The following hold true:
\begin{enumerate}
\item Assume that\/ $\operatorname{p.dim}_R R_S\leq 1$.
        Then an $R$\+module $C$ is $S$\+divisible if and only if
        it is annihilated by $I$ and\/ $\Ext^1_R(R_S/\phi(R), C)=0$.
        Moreover, every $S$-divisible module is $S$-weakly cotorsion.
\item If $R_S$ is a perfect ring, then the class\/
        $\clP_1^{\perp_1}\subset \Modl R$ contains
        the class of $S$\+h-divisible $R$\+modules.
\end{enumerate}
\end{lem}

\begin{proof}
 (1) For every $R$\+module $C$, we put $C[I]=\{x\in C\mid Ix=0\}$.
 The exact sequence $0\to R/I\to R_S\to R_S/\phi(R)\to 0$ induces
an exact sequence
\begin{multline}
\tag{$\ast$}
 0\la \Hom_R(R_S/\phi(R)), C)\la \Hom_R(R_S, C)\la C[I] \\
 \la \Ext^1_R(R_S/\phi(R)), C)\la \Ext^1_R(R_S, C).
\end{multline}
 If $C$ is an $S$\+divisible $R$\+module, then,
by Proposition~\ref{P:Matlis}, $C[I]=C$ and $C$ is $S$\+h-divisible,
hence $\Hom_R(R_S, C)\to C$ is an epimorphism.
 For every $R_S$\+module $Y$, we have $\Ext^1_R(R_S, Y)=0$ and
thus $\Ext^1_R(R_S, C)=0$, since $\operatorname{p.dim}_R R_S\leq 1$.
 In particular, $C$ is $S$-weakly cotorsion and from sequence~($\ast$)
we conclude that $\Ext^1_R(R_S/\phi(R), C)=0$.
 
 Conversely, assume that $C$ is an $R/I$-module such that 
$\Ext^1_R(R_S/\phi(R), C)\allowbreak=0$.
 Then, by sequence~($\ast$), $\Hom_R(R_S, C)\to C$ is an epimorphism,
that is $C$ is $S$\+divisible.
 
 (2)  Let $M\in \clP_1(R)$, and let $Y$ be an $R_S$-module.
 Then $\Ext^1_R(M, Y)\cong \Ext^1_{R_S}(M\otimes_RR_S, Y)=0$,
since $M\otimes_RR_S\in\clP_1(R_S)$ and therefore $M\otimes_RR_S$
is a projective $R_S$-module (the ring $R_S$ is perfect).
 Every $S$\+h-divisible $R$\+module $D$ is an epimorphic image of
a free $R_S$\+module, thus also $\Ext^1_R(M, D)=0$.
 \end{proof}

\begin{lem}\label{L:S-torsion} 
 The following hold true:
\begin{enumerate} 
\item For every $R$\+module $M$, the $R$\+module
$\Tor_1^R(M, R_S/\phi(R))$ is isomorphic to\/ $\Gamma_S(M)/IM$.
\item If $F$ is a flat $R$\+module, then $IF=\Gamma_S(F)$.
\end{enumerate}
\end{lem}

\begin{proof}
 (1) Tensoring by $M$ the short exact sequence
$0\to R/I\to R_S \to R_S/\phi(R)\to 0$, we obtain the exact sequence
\[
 0\la \Tor_1^R(M, R_S/\phi(R))\la M/IM\la M\otimes_RR_S.
\]
 The conclusion follows recalling that the maximal $S$\+torsion submodule
$\Gamma_S(M)\subset M$ is the kernel of the morphism $M\to M\otimes_RR_S$.
 
 (2) Follows from~(1).
\end{proof}

\begin{lem}\label{L:torsion-Hom}
 The following hold true:
\begin{enumerate}
\item If $N$ is an $S$\+torsion $R$\+module, then, for every
$R$\+module $X$, the $R$\+module\/ $\Hom_R(N, X)$ is
an $S$\+contramodule (i.e., $S$\+h-reduced and $S$\+weakly cotorsion).
\item If $Q$ is an $S$\+contramodule $R$\+module, then, for every
$R$\+module $Y$, the $R$\+module\/ $\Hom_R(Y,Q)$ is also
an $S$\+contramodule.
\end{enumerate}
\end{lem}
 
\begin{proof}
 (1) By adjunction, $\Hom_R(R_S, \Hom_R(N, X))\cong
\Hom_R(R_S\otimes N, X)$, and $R_S\otimes N=0$ by assumption,
hence $\Hom_R(N, X)$ is $S$\+h-reduced.
 
 Consider an injective $R$\+module $E$ containing~$X$.
 Then for every $n\ge1$ we have $\Ext^n_R(R_S, \Hom_R(N, E))\cong
\Hom_R(\Tor^R_n(R_S, N), E)=0$ and an exact sequence
 \[0\la \Hom_R(N, X)\la \Hom_R(N, E)\overset{\nu}\la \Hom_R(N, E/X).\]
  By the above argument, $\Img \nu$ is an $S$\+h-reduced $R$\+module,
hence we obtain:
\begin{multline*}
 0=\Hom_R(R_S, \Img \nu)\la \Ext^1_R(R_S,\Hom_R(N, X)) \\
 \la \Ext^1_R(R_S, \Hom_R(N, E))=0.
\end{multline*}

 (2) Presenting $Y$ as the cokernel of a morphism of free $R$\+modules
$E\to F$, we see that $\Hom_R(Y,Q)$ is the kernel of the induced morphism
$\Hom_R(F,Q)\to\Hom_R(E,Q)$ between two products of copies of~$Q$.
 Since the full subcategory of $S$\+contramodule $R$\+modules is closed
under the kernels and products in $\Modl R$ by Lemma~\ref{S-contra}~(3),
the assertion follows.
\end{proof}

\begin{lem} \label{L:regular-in-R_S}
 Let $\Sigma$ denote the set of all regular elements of $R$.
 Then the following hold true:
\begin{enumerate}
\item $\Sigma_S$ is a subset of the set of all regular elements
      of~$R_S$.
\item If $R_S$ is a perfect ring, then $R_S$ coincides with
      its total ring of quotients $Q(R_S)$.
\item If $R_S$ is a perfect ring and $R/sR$ is a perfect ring
      for every $s\in S$, then $R/rR$ is a perfect ring for every
      regular element $r\in R$.
\end{enumerate}
\end{lem}
 
\begin{proof}  
 (1)
 If $r$ is a regular element in $R$ and $rb/s=0$ in $R_S$, then
there is an element $t\in S$ such that $trb=0$ in $R$,
hence $tb=0$ in $R$, that is $b/s=0$ in $R_S$.
 
 (2) is an immediate consequence of Lemma~\ref{L:perfect}.
 
 (3) If $r\in R$ is a regular element, then by~(1) and~(2) there is
an element $s\in S$ such that $s\in rR$.
 Then $R/rR$ is a perfect ring, being a quotient ring of~$R/sR$.
\end{proof}

\section{$S$-strongly flat modules} \label{strongly-flat-secn}
%

 
 In this section we consider the problem of the existence of
$S$\+strongly flat covers and show that if every $R$-module admits
an $S$\+strongly flat cover then the ring $R_S$ is perfect and for
every $s\in S$ all the rings $R/sR$ are perfect.

\begin{lem}\label{L:S-strongly}
 Assume that $N$ is an $S$\+strongly flat $R$-module.
 Then following hold:
\begin{enumerate}
\item $N\otimes_RR_S$ is a projective $R_S$\+module.
\item For every $s\in S$, $N/sN$ is a projective $R/sR$\+module.
\item If $N$ is moreover $S$-divisible, then $N$ is a projective
$R_S$\+module.
\end{enumerate}
\end{lem}

\begin{proof} (1) and (2) follow easily from Lemma~\ref{L:SF} tensoring
by $R_S$ and by $R/sR$ the pure exact sequence~(SF).

(3) Assume that $N$ is $S$\+divisible.
 Then $IN=0$ by Proposition~\ref{P:Matlis}~(1),
hence, by Lemma~\ref{L:S-torsion}~(2), $N$ is $S$\+torsion free.
 This means that $N$ is isomorphic to $N\otimes_RR_S$,
hence $N$ is a projective $R_S$\+module by part~(1).
\end{proof}

 We introduce a ``restricted'' notion of superfluous submodule which
will be useful in the characterization of covers.

\begin{defn}  Let $(\A, \B)$ be a cotorsion pair in $\ModR$.
 Let $A$ be a right $R$-module in the class $\A$ and let $B\in \B$
be a submodule of~$A$.

 We say that $B$ is \emph{$\B$-superfluous} in $A$, and
write $B\prec  A$, if for every submodule $H$ of $A$,
$H+B=A$ and $H\cap B\in \B$ imply $H=A$.
\end{defn}

 The following fact will be used throughout in the sequel;
its proof is straightforward.

\begin{lem}\label{L:endomorph}
 Let $0\to B\hookrightarrow A\stackrel{\psi}\to M\to 0$
be a short exact sequence of modules, and let $f$ be an endomorphism
of $A$ such that $\psi = \psi f$. Then $f(A)+B=A$, $\Ker f\subseteq
B$ and $B\cap f(A)=f(B)$.
\end{lem}

 The next proposition shows how the notion of $\B$-superfluous
submodules is related to $\A$-covers.

\begin{prop}\label{P:superf} Let 
\begin{equation}
 0\to B\hookrightarrow A\stackrel{\psi}\to M\to 0 \tag{$\ast$}
\end{equation}
be a special $\A$-precover of the $R$-module $M$. The
following hold:
\begin{enumerate}
\item if {\rm ($*$)} is an $\A$-cover of $M$, then $B\prec A$;
\item if $M$ admits an $\A$-cover and $B\prec A$,
then {\rm ($*$)} is an $\A$-cover.
\end{enumerate}
\end{prop}

\begin{proof}
 (1) Let $H\subseteq A$ be such that $H+B=A$ and $H\cap B\in\B$.
Consider the diagram:
\[\xymatrix{
0\ar[r]&
{H\cap B}\ar[r] &H\ar[r]^{\psi_{|H}} & M\ar[r]&0\\
&& A\ar@{-->}[u]^{\phi} \ar[ur]_{\psi}}
\]
where $\psi_{|H}$ is the restriction of~$\psi$ to~$H$.
 The diagram can be completed by~$\phi$,
since $\Ext_R^1(A, H\cap B)=0$.
 Consider the inclusion~$\varepsilon$ of $H$ into $A$;
then by the above diagram it is clear that
$\psi=\psi\varepsilon\phi$.
 Since {\rm ($*$)} is an $\A$\+cover of $M$, we conclude that
$\varepsilon\phi$ is an automorphism of $A$, hence $H=A$.

(2) By Xu's result \cite[Corollary 1.2.8]{Xu}, it is enough to show that
if $B\prec A$, then $B$ doesn't contain any nonzero summand of~$A$.
 Assume $A=Y\oplus A_1$ with $Y\subseteq B$; then $B+A_1=A$ and
$B=Y\oplus (B\cap A_1)$.
 The module $B\cap A_1$ is in $\B$, being a summand of $B$;
hence $A_1=A$, since $B\prec A$.
\end{proof}

 We apply the previous results to the cotorsion pair
$(S$\+$\SF, S$\+$\WC)$.

\begin{lem}\label{L:divisible-cover}
 Let $D$ be an $S$-divisible module.  Assume that
\[0\to C\hookrightarrow A\stackrel{\psi}\to D\to 0\]
is an $S$\+$\SF$\+cover of $D$.
 Then $A$ is $S$\+divisible and $D$ is $S$\+$h$-divisible.
\end{lem}

\begin{proof}
 Let $s\in S$. We have $sA+C=A$, since $D$ is $S$\+divisible.
 From the exact sequence $0\to C\cap sA\to C\to A/sA\to 0$ we get
$$
 0=\Hom_R(R_S, A/sA)\la \Ext^1_R(R_S, C\cap sA)\la
 \Ext^1_R(R_S, C)=0,
$$
hence $C\cap sA$ is $S$\+weakly cotorsion.
 By Proposition~\ref{P:superf}~(1), $sA=A$, thus $A$ is $S$-divisible.
 Now by Lemma~\ref{L:S-strongly} $A$ is a projective $R_S$-module,
hence $D$ is $S$\+$h$-divisible.
\end{proof}

\begin{lem}\label{L:R_S-modules}
 Assume that every $R_S$-module admits an  $S$\+$\SF$\+cover.
 Then $R_S$ is a perfect ring.
\end{lem}

\begin{proof}
 Let $M$ be an $R_S$-module and let
$(a)\quad 0\to C\hookrightarrow A\stackrel{\psi}\to M\to 0$
be an $S$\+$\SF$\+cover of~$M$.
 $M$~is $S$-divisible, hence, by Lemmas~\ref{L:divisible-cover}
and~\ref{L:S-strongly}, $A$ is a projective $R_S$-module.
 Tensoring by $R_S$ we conclude that $(a)$ is an exact sequence
of $R_S$\+modules.
 Since $(a)$ was a cover and $R\to R_S$ is a flat ring epimorphism,
we have that every $R_S$\+endomorphism $f$ of $A$ such that
$\psi\circ f=\psi$ is an automorphism of $A$.
 Thus, every $R_S$\+module admits a projective cover, which
means that $R_S$ is a perfect ring.
\end{proof}

\begin{lem}\label{L:R/sR-modules}
 Let $s\in S$ and assume that every $R/sR$-module admits
an $S$\+$\SF$\+cover.
 Then $R/sR$ is a perfect ring.
\end{lem}

\begin{proof}
 Let $M$ be an $R/sR$-module and let
$(a)\quad 0\to C\hookrightarrow A\stackrel{\psi}\to M\to 0$
be an $S$\+$\SF$\+cover of~$M$.
 We may assume, w.l.o.g., that $M$ is not a projective $R/sR$\+module
(otherwise $R/sR$ is semisimple, hence perfect).

 By Lemma~\ref{L:S-strongly}~(2), $A/sA$ is a projective $R/sR$\+module.
 We have $sA\subseteq C$ and $sA\varsubsetneq C$ by the assumption
on~$M$.
 We show that the exact sequence
\begin{equation}
 0\la C/sA\la A/sA\la M\la 0 \tag{$b$}
\end{equation}
is a projective $R/sR$\+cover of $M$, so that we conclude that
$R/sR$ is a perfect ring.
 Let $H$ be a submodule of $A$ containing $sA$ and such that
$C/sA+H/sA=A/sA$.
 Then $C+H=A$ and the exact sequence $0\to C\cap H\to C\to A/H\to 0$
shows that $C\cap H$ is an $S$\+weakly cotorsion module,
since $C$ is $S$\+weakly cotorsion and $sA\subseteq H$.
 By Proposition~\ref{P:superf}, $H=A$, hence $C/sA$ is a superfluous
submodule of $A/sA$.
\end{proof}

\section{$S$-$\mathrm h$-local rings}
\label{S-h-local-secn}


 In this section we define $S$\+h-local rings and we show that they are characterized by the fact  every $S$\+torsion module decomposes as a direct sum of its localizations at the maximal ideals of $R$. Moreover, with the extra assumption of the boundness of the $S$\+torsion of $R$, we prove a dual characterization, namely $R$ is $S$\+h-local if and only if every $S$\+contramodule is a direct product  of its colocalizations at the maximal ideals of~$R$.
 
  We denote by $\Max R$ the set of all maximal ideals of the ring~$R$.
 \begin{defn}
 We say that $R$ is \emph{$S$\+h-local} if every element $s\in S$ is
contained only in finitely many maximal ideals of $R$ and every prime
ideal of $R$ intersecting $S$ is contained in only one maximal ideal.
\end{defn}

\begin{lem} \label{two-maximal-ideals}
 Let $S$ be a multiplicative subset of a commutative ring $R$ such that
every prime ideal of $R$ intersecting $S$ is contained in only one
maximal ideal.
 Let\/ $\m\ne\n$ be two maximal ideals of~$R$.
 Then the images of all elements of $S$ are invertible in the ring
$R_\m\otimes_R R_\n$.
\end{lem}

\begin{proof}
 Let $s\in S$ be an element.
 Assume that $\q$~is a prime ideal of $R_\m\otimes R_\n$.
 Then $\q$~is the localization of a prime ideal $\p$ contained in
$\m\cap\n$.
 By assumption, $\p$~does not intersect $S$, so $s\notin\p$.
 Thus no prime ideal of $R_\m\otimes R_\n$ contains~$s$, which implies
that $s$~is invertible in $R_\m\otimes R_\n$.
\end{proof}

\begin{prop}\label{P:h-local}
 Let $S$ be a multiplicative subset of a commutative ring~$R$.
 Then the following are equivalent:
\begin {enumerate}
\item $R$ is $S$\+h-local.
\item Every $S$\+torsion $R$\+module $M$ is isomorphic to the direct sum\/
        $\bigoplus\limits_{\m\in \Max R}M_{\m}$.
\item Every $S$\+torsion $R$\+module $M$ is isomorphic to the direct sum\/
        $\bigoplus\limits_{\m\in \Max R;\,\m\cap S\neq\varnothing}M_{\m}$.
\end{enumerate}
\end{prop}

\begin{proof}
(1) $\Rightarrow$ (2) We first show that for every $s\in S$,
$R/sR$ is isomorphic to the finite direct sum
$\bigoplus\limits_{s\in\m\in \Max R}(R/sR)_{\m}$. 
 By assumption $s$~is contained only in finitely many maximal ideals
$\m_i$, $i=1, \dots, n$, hence there is a monomorphism
$\mu\colon R/sR\to \bigoplus\limits_{i=1,\dots,n}(R/sR)_{\m_i}$.
 To prove that $\mu$~is surjective we show that the localizations of
$\mu(R/sR)$ and of $\bigoplus\limits_{i=1,\dots,n}(R/sR)_{\m_i}$
at every maximal ideal~$\n$ of $R$ are equal.
 This is clear if $\n$~is a maximal ideal different from $\m_i$,
for every $i=1,\dots,n$.
 It remains to observe that $(R/sR)_{\m_i}\otimes R_{\m_j}=0$ for
every $i\neq j$ in $1,\dots,n$, since $s$~is a unit in
$R_{\m_i}\otimes R_{\m_j}$ by Lemma~\ref{two-maximal-ideals}.

 Let now $M$ be an $S$\+torsion module and let $\psi\colon M\to
\prod\limits_{\m\in \Max R}M_{\m}$.
 Then $\psi$~is a monomorphism and the image of~$\psi$ is contained
in the direct sum $\bigoplus\limits_{\m\in \Max R}M_{\m}$.
 Indeed, let $x\in M$ and $s\in S$ be such that $xs=0$.
 By assumption $s$~is contained only in finitely many maximal ideals,
hence $x$~is zero in $M_{\m}$ for every maximal ideal~$\m$
not containing~$s$.
 By the previous argument $R/sR$ is a direct sum of finitely many local
rings for every $s\in S$, hence it follows that the map
$M\to\bigoplus\limits_{\m\in \Max R}M_{\m}$ is surjective.
 
(2) $\Leftrightarrow$ (3) is obvious, since for every $S$\+torsion
module $M$, $M_{\m}=0$ for every maximal ideal~$\m$ not
intersecting~$S$.

(2) $\Rightarrow$ (1) Condition (2) applied to $R/sR$ easily implies that
every element $s\in S$ is contained only in finitely many maximal ideals.

 Let $\p$ be a prime ideal of $R$ containing an element $s\in S$ and
assume that $\p$~is contained in two different maximal ideals $\m$
and~$\n$.
 Then $R/\p$ is $S$\+torsion and $R_\p=R_\p\otimes R_\m=R_\p\otimes R_\n$.
 By~(2), we obtain that $R_\p/\p R_\p$ contains the direct sum of two copies of
$R_\p/\p R_\p$, a contradiction.
\end{proof}

 The direct sum decomposition in Proposition~\ref{P:h-local} is unique and
functorial in a strong sense described by the next lemma.

\begin{lem} \label{sum-decomp-unique}
 Let $R$ be an $S$\+h-local ring.
 Let $M(\m)$ and $N(\m)$ be two collections of $R_\m$\+modules
indexed by\/ $\m\in\Max R$.
 Assume that the modules $M(\m)$ are $S$\+torsion.
 Then every $R$\+module morphism\/ $\bigoplus_\m M(\m)\to
\bigoplus_\m N(\m)$ is a direct sum of $R_\m$\+module morphisms\/
$M(\m)\to N(\m)$.
\end{lem}
 
\begin{proof}
 It suffices to check that $\Hom_R(M(\m),N(\n))=0$ for any two maximal
ideals $\m\ne\n$.
 Indeed, $M(\m)\otimes R_\n$ is an $(R_\m\otimes R_\n)$-module,
and by Lemma~\ref{two-maximal-ideals} all elements of $S$ are invertible
in $R_\m\otimes R_\n$.
 Since $M(\m)\otimes R_\n$ is an $S$\+torsion module, it follows
that $M(\m)\otimes R_\n=0$.
\end{proof}
 
 The above results yield the following equivalence of categories.

\begin{cor} 
For any $S$\+h-local ring $R$, the direct sum
decomposition of Proposition~\ref{P:h-local} establishes an equivalence
between the abelian category of $S$\+torsion $R$\+modules and
the Cartesian product of the abelian categories of $S$\+torsion
$R_\m$\+modules over all\/ $\m\in\Max R$, \,$\m\cap S\ne\varnothing$.
\end{cor}

 In the rest of this section, our aim is to obtain a dual version of
Proposition~\ref{P:h-local} for direct product decompositions of
$S$\+contramodule $R$\+modules (under certain assumptions).
 We start with a uniqueness/functoriality lemma dual
to Lemma~\ref{sum-decomp-unique}.

\begin{lem} \label{product-decomp-unique}
 Let $R$ be an $S$\+h-local ring.
 Let $P(\m)$ and $Q(\m)$ be two collections of $R_\m$\+modules
indexed by\/ $\m\in\Max R$.
 Assume that the modules $Q(\m)$ are $S$\+h-reduced.
 Then every $R$\+module morphism\/ $\prod_\m P(\m)\to\prod_\m Q(\m)$
is a product of $R_\m$\+module morphisms $P(\m)\to Q(\m)$.
\end{lem}

\begin{proof}
 Notice first of all that if $\n\in\Max R$ does not intersect $S$, then
all $R_\n$\+modules are $R_S$\+modules, so $Q(\n)$ is simultaneously
an $R_S$\+module and $S$\+h-reduced, hence $Q(\n)=0$.

 It suffices to check that
$\Hom_R\bigl(\prod_{\m\ne\n}P(\m),Q(\n)\bigr)=0$ for all $\n\in\Max R$.
 Consider the short exact sequence of $R$\+modules
$$
 0\la \bigoplus_{\m\in\Max R;\,\m\ne\n}P(\m)\la
 \prod_{\m\in\Max R;\,\m\ne\n}P(\m)\la D\la0.
$$
 For any $s\in S$, the set $\{\m\in\Max R\mid s\in\m\}$ is finite, so we
have a short exact sequence
$$
 0\la \bigoplus_{\m\in\Max R;\,s\notin\m}P(\m)\la
 \prod_{\m\in\Max R;\,s\notin\m}P(\m)\la D\la0
$$
with the same quotient module~$D$.
 Since the action of~$s$ is invertible in $P(\m)$ when $s\notin\m$,
it follows that $s$~acts invertibly in~$D$.
 So $D$ is an $R_S$\+module, and consequently
$\Hom_R(D,Q)=0$ for any $S$\+h-reduced $R$\+module~$Q$.

 Thus the map $\Hom_R\bigl(\prod_{\m\ne\n}P(\m),Q(\n)\bigr)\to
\Hom_R\bigl(\bigoplus_{\m\ne\n}P(\m),Q(\n)\bigr)$ is injective, and
it remains to show that $\Hom_R(P(\m),Q(\n))=0$ for all $\m\ne\n$.
 Now $P(\m)\otimes R_\n$ is an $(R_\m\otimes R_\n)$-module, hence it is
also an $R_S$\+module.
 Therefore, $\Hom_R(P(\m),Q(\n))=\Hom_R(P(\m)\otimes R_\n,\>Q(\n))=0$.
\end{proof}


\begin{nota}\label{N:notation} 
Denote by $K^\bullet=K^\bullet_{R,S}$ the two-term complex of
$R$\+modules $R\stackrel\phi\to R_S$, where the term $R$ is placed
in the cohomological degree~$-1$ and the term $R_S$ is placed in
the cohomological degree~$0$.
 We will view this complex as an object of the derived category
$\Der{\Modl R}$.
\end{nota}
 When all the elements of $S$ are nonzero-divisors in $R$, one can use
the conventional quotient module $K=R_S/R$ in lieu of
the complex~$K^\bullet$.

 We will use the simplified notation
$$
 \Ext_R^n(K^\bullet,M)=\Hom_{\Der{\Modl R}}(K^\bullet,M[n]),
 \qquad n\in\mathbb Z
$$
for any $R$\+module~$M$.
 Notice that one may have $\Ext^1_R(K^\bullet,J)\neq0$ for
an injective $R$\+module $J$ when there are zero-divisors in~$S$.

 We have a distinguished triangle in the derived category $\Der{\Modl R}$
$$
 R\stackrel\phi\la R_S\la K^\bullet\la R[1].
$$
 The induced long exact sequence of modules
$\Hom_{\Der{\Modl R}}({-},M[*])$ includes such fragments as
\begin{multline} \tag{$\ast$}
 0=\Ext_R^{n-1}(R,M)\la\Ext_R^n(K^\bullet,M) \\
 \la\Ext_R^n(R_S,M)\la\Ext_R^n(R,M)=0 \quad\text{for $n<0$ or $n>1$}
\end{multline}
and
\begin{multline} \tag{$\ast\ast$}
 0\la\Ext^0_R(K^\bullet,M)\la\Hom_R(R_S,M) \la M \\
 \la \Ext^1_R(K^\bullet,M)\la\Ext^1_R(R_S,M)\la0.
\end{multline}

\begin{lem} \label{ext-from-K}
 For any $R$\+module $M$ one has:
\begin{enumerate}
\item $\Ext^n_R(K^\bullet,M)=0$ for $n<0$;
\item $\Ext^0_R(K^\bullet,M)\cong\Hom_R(R_S/\phi(R),M)$;
\item $\Ext^n_R(K^\bullet,M)\cong\Ext^n_R(R_S,M)$ for $n>1$.
\end{enumerate}
\end{lem}

\begin{proof}
 The vanishing~(1) and the isomorphism~(3) follow from
the exact sequence~($\ast$).
 The isomorphism~(2) follows from the exact sequence~($\ast\ast$).
\end{proof}

 We will denote by $\Delta(M)=\Delta_{R,S}(M)$ the $R$\+module
$\Ext^1_R(K^\bullet,M)$.
 For any $R$\+module $M$, we have a natural $R$\+module morphism
$\delta_M\colon M\to\Delta(M)$.

\begin{lem} \label{S-contra-Delta}
 Let $S$ be a multiplicative subset of a commutative ring~$R$.
 Then an $R$\+module $P$ is an $S$\+contramodule if and only if
there exists an $S$\+h-reduced $R$\+module $M$ such that $P$ is
isomorphic to $\Delta(M)$.
\end{lem}

\begin{proof}
 For any $S$\+h-reduced $R$\+module $M$, the $R$\+module $\Delta_S(M)$
is an $S$\+con\-tramodule by~\cite[Lemma~1.7(b)]{Pos}.
 Conversely, if $P$ is an $S$\+contramodule $R$\+module, then $P$ is
$S$\+h-reduced and the morphism $\delta_P\colon P\to\Delta(P)$ is
an isomorphism in view of the exact sequence~($\ast\ast$).
\end{proof}

 We will say that \emph{the $S$\+torsion in $R$ is bounded} if there
exists an element $s_0\in S$ such that $s_0I=0$ in~$R$.

\begin{lem} \label{K-direct-sum-decomp}
 Let $R$ be an $S$\+h-local ring such that the $S$\+torsion in $R$
is bounded.
 Then the complex $K^\bullet$, as an object of the derived category\/
$\Der{\Modl R}$, is isomorphic to the direct sum\/
$\bigoplus\limits_{\m\in\Max R;\,\m\cap S\neq\varnothing} K^\bullet_\m$.
\end{lem}

\begin{proof}
 There are many ways to obtain this isomorphism, which in fact holds
for any complex of $R$\+modules $L^\bullet$ with $S$\+torsion
cohomology modules~\cite[Remark~13.4]{Pcta}.
 An explicit construction (specific to the complex $K^\bullet=
K^\bullet_{R,S}$) can be found in~\cite[proof of Lemma~13.3]{Pcta}.

 Alternatively, notice that two-term complexes of $R$\+modules
$M^\bullet= (M^{-1}\stackrel{f}\to M^0)$ with the cohomology modules
$H^{-1}(M^\bullet)=\Ker(f)$ and $H^0(M^\bullet)=\Coker(f)$
are classified, up to isomorphism in $\Der{\Modl R}$, by elements
of the group/module $\Ext^2_R(\Coker(f),\Ker(f))$.
 In the case of the complex $K^\bullet$, we have $H^{-1}(K^\bullet)=I$
and $H^0(K^\bullet)=R_S/\phi(R)$.

 Let $\n_1,\dots,\n_k$ be the maximal ideals of $R$ containing~$s_0$.
 Then, by Proposition~\ref{P:h-local}, we have
$I\cong\bigoplus\limits_{j=1}^k I_{\n_j}$ and $R_S/\phi(R)\cong
\bigoplus\limits_{\m\in\Max R;\,\m\cap S\neq\varnothing}
(R_S/\phi(R))_\m$.
 Now, for any maximal ideals $\m\ne\n$ and any $R_\m$\+module $A$,
the $R$\+modules $\Ext_R^i(A,I_\n)$ are $R_\m\otimes R_\n$\+modules
annihilated by~$s_0$.
 Hence $\Ext_R^i(A,I_\n)=0$ for all $i\ge0$, and the desired direct
sum decomposition follows.
\end{proof}

 For any $R$\+module $P$ and any maximal ideal~$\m$ of the ring $R$,
we denote by $P^\m$ the colocalization $\Hom_R(R_\m,P)$.

\begin{prop} \label{product-decomp-exists}
 Let $S$ be a multiplicative subset of a commutative ring~$R$.
 Consider the following conditions:
\begin{enumerate}
\item $R$ is $S$\+h-local.
\item Every $S$\+contramodule $R$\+module $P$ is isomorphic to
        the direct product\/ $\prod\limits_{\m\in\Max R}P^\m$.
\item Every $S$\+contramodule $R$\+module $P$ is isomorphic to
        the direct product\/
        $\prod\limits_{\m\in\Max R;\,\m\cap S\ne\varnothing}P^\m$.
\end{enumerate}
 The implications {\rm (1)} $\Leftarrow$ {\rm (2)}
$\Leftrightarrow$ {\rm (3)} hold true.
 Assuming that the $S$\+torsion in $R$ is bounded, all the three
conditions are equivalent.
\end{prop}

\begin{proof}
 (2) $\Leftrightarrow$ (3) By Lemma~\ref{L:torsion-Hom}~(2),
the $R^\m$\+module $P^\m$ is an $S$\+contramodule for any
$S$\+contramodule $R$\+module $P$ and any maximal ideal~$\m$ of~$R$.
 When $\m\cap S=\varnothing$, any $R_\m$\+module is
an $R_S$\+module; so $P^\m=0$.

 (2) $\Rightarrow$ (1) It is easy to see that any $R$\+module
annihilated by an element of $S$ is an $S$\+contramodule.
 To prove that any prime ideal~$\p$ of $R$ intersecting $S$ is contained
in only one maximal ideal, consider the $R$\+module $P=R_\p/\p R_\p$.
 Since $P$ is annihilated by an element of $S$, it is
an $S$\+contramodule, so by~(2) we have
$P\cong\prod\limits_{\m\in\Max R}P^\m$.
 Now, if $\p\subseteq\m$ then $P$ is an $R_\m$\+module,
hence $P^\m\cong P$.
 As $P$ is a field and an indecomposable $R$\+module, there cannot be
two maximal ideals in $R$ containing~$\p$.

 Let $s\in S$ be an element.
 To prove that $s$~is only contained in finitely many maximal ideals
of $R$, we apply~(2) to the $R$\+module $P=R/sR$.
 We largely follow the argument in~\cite[proof of Theorem~3.1]{Mat}.
 Assume that $P\cong\prod\limits_{\m\in\Max R} P(\m)$, where $P(\m)$
are some $R_\m$\+modules.
 Denote by $Q$ the submodule $\bigoplus_\m P(\m)\subseteq P$.
 For every fixed $\m\in\Max R$ we have $P=P(\m)\oplus P'(\m)$ and
$Q=P(\m)\oplus Q'(\m)$, where $P'(\m)$ (respectively, $Q'(\m)$) denotes
the product (resp., the direct sum) of $P(\n)$ over all
$\n\in\Max R$, \,$\n\neq\m$.

 For any $\n\ne\m$, we have $P(\n)^\m=0$, because $P(\n)^\m$ is
an $(R_\m\otimes R_\n)$-module annihilated by~$s$, and
$s$ is invertible in $R_\m\otimes R_\n$ by
Lemma~\ref{two-maximal-ideals}.
 Therefore, $Q'(\m)^\m\subseteq P'(\m)^\m=0$, \ $P(\m)\cong
P(\m)^\m\cong P^\m$, and $Q^\m\cong P(\m)$.
 Applying~(2) to the $R$\+module $Q$ (which is annihilated by~$s$
as a submodule of $P$, hence also an $S$\+contramodule), we see that
$Q$ is isomorphic to $\prod\limits_{\m\in\Max R} P(\m)$, so $Q$ is
isomorphic to~$P$.

 Now the $R$\+module $P=R/sR$ is finitely presented, $Q\cong P$,
and $Q\cong\bigoplus_\m P(\m)$, hence $P(\m)=0$ for all but a finite
number of maximal ideals~$\m$ of the ring~$R$.
 If $P(\n)=0$ for some $\n\in\Max R$, then $P_\n=0$, since $P(\m)_\n$
is an $(R_\m\otimes R_\n)$-module annihilated by~$s$ for $\m\ne\n$,
so $P(\m)_\n=0$.
 Thus there are only finitely many maximal ideals $\m$ in $R$ for
which $P_\m\ne0$.
 Since $P_\m\simeq R_\m/sR_\m\ne0$ if and only if $s\in\m$, we
are done.

 (1) $\Rightarrow$ (3) Let $P$ be an $S$\+contramodule $R$\+module.
 By Lemma~\ref{S-contra-Delta}, there exists an $S$\+h-reduced
$R$\+module $M$ such that $P\cong\Delta(M)=\Ext_R^1(K^\bullet,M)$
(in fact, one can take $M=P$).
  By Lemma~\ref{K-direct-sum-decomp}, we have a direct sum
decomposition $K^\bullet\cong\bigoplus_\m K^\bullet_\m$ in
$\Der{\Modl R}$.
 So it remains to show that the $R$\+module $\Ext_R^1(K^\bullet_\m,M)$
is isomorphic to $\Delta(M)^\m$.

 Indeed, for any $R$\+module $N$ the derived adjunction isomorphism
$$
 \mathbf{R}\!\Hom(R_\m,\,\mathbf{R}\!\Hom_R(K^\bullet,N))\simeq
 \mathbf{R}\!\Hom_R(K^\bullet\otimes_RR_\m,\>N)
$$
leads to a spectral sequence
$$
 E_2^{p,q}=\Ext_R^p(R_\m,\Ext_R^q(K^\bullet,N))\Longrightarrow
 E_\infty^{p,q}=\mathrm{gr}^p\Ext^{p+q}_R(K^\bullet_\m,N)
$$
with the differentials $d_r^{p,q}\colon E_r^{p,q}\to E_r^{p+r,q-r+1}$.
 In low degrees, it reduces to an exact sequence
\begin{multline*}
 0\la\Ext^1_R(R_\m,\Hom_R(K^\bullet,N)) \\ \la
 \Ext^1_R(K^\bullet_\m,N)\la\Hom_R(R_\m,\Ext^1_R(K^\bullet,N))
 \\ \la\Ext^2_R(R_\m,\Hom_R(K^\bullet,N))
 \la\Ext^2_R(K^\bullet_\m,N).
\end{multline*}

 For an $S$\+h-reduced $R$\+module $M$, we have $\Hom_R(K^\bullet,M)=0$,
hence the desired isomorphism $\Ext^1_R(K^\bullet_\m,M)\cong
\Hom_R(R_\m,\Ext_R^1(K^\bullet,M))$.
\end{proof}

 The next corollary follows from Lemma~\ref{product-decomp-unique} and
Proposition~\ref{product-decomp-exists}.

\begin{cor} For any $S$\+h-local
ring $R$ with bounded $S$\+torsion, the category of $S$\+contramodule
$R$\+modules is equivalent to the Cartesian product of the categories
of $S$\+contramodule $R_\m$\+modules over all\/
$\m\in\Max R$, \,$\m\cap S\ne\varnothing$.
\end{cor}

\section{$t$-contramodules}
\label{t-contra-secn}

 In this section, $R$ will denote a commutative ring and $t\in R$ will
be its element.

 An $R$\+module $M$ is \emph{$t$\+torsion} if for every element $x\in M$
there exists an integer $n\ge 1$ such that $t^nx=0$.
 Clearly, for any fixed $R$\+module $M$, the set of all $t\in R$ such
that $M$ is a $t$\+torsion $R$\+module is a radical ideal in~$R$
(i.e., an ideal that coincides with its radical).

 Denote by $R[t^{-1}]$ the localization of the ring $R$ at
the multiplicative subset $S_t=\{1,t,t^2,\dotsc\}\subset R$
generated by~$t$.
 An $R$\+module $P$ is called a \emph{$t$\+contramodule} if
$\Hom_R(R[t^{-1}],P)=0=\Ext^1_R(R[t^{-1}],P)$.

\begin{lem} \label{t-contra}
 Let $t$ be an element in a commutative ring~$R$.
 Then the following hold true:
\begin{enumerate}
\item The full subcategory of all $t$\+contramodule $R$\+modules
is closed under the kernels, cokernels, extensions, and infinite
products in the category $\Modl R$. 
\item If $P$ is a $t$\+contramodule $R$\+module, then
the element\/ $1-t\in R$ acts by an automorphism of~$P$.
\item For any $R$\+module $P$, the set of all elements
$t\in R$ such that $P$ is a $t$\+contramodule is a radical ideal
in the ring~$R$.
\end{enumerate}
\end{lem}

\begin{proof}
 (1) is Lemma~\ref{S-contra}~(3--4) applied to the multiplicative subset
$S_t=\{1,t,t^2,\dotsc\}\subset R$.

 (2) Let $K$ and $L$ be the kernel and cokernel of the map
$1-t\colon P\to P$.
 By~(1), $K$ and $L$ are $t$\+contramodule $R$\+modules.
 However, the element $t\in R$ acts by the identity automorphisms
of $K$ and~$L$, so $K\cong\Hom_R(R[t^{-1}],K)$ and similarly for~$L$.
 Hence $K=L=0$.
 (For another argument, see~\cite[Theorem~3.3(c) and
Remark~9.2]{Pcta}.)

 (3) is~\cite[Theorem~5.1 and Remark~5.5]{Pcta} (see
also~\cite[paragraph after Theorem~7.3]{Pcta}).
\end{proof}

 In the rest of this section, our aim is to discuss the following
result, which we will use in Section~\ref{S-almost-perfect-secn}.
 It can be deduced as a particular case of~\cite[Lemma~9.4 and
Theorem~9.5]{Pcta} (corresponding to the case when the finitely 
generated ideal considered there is, in fact, principal and generated
by~$t$) or as a corollary of~\cite[Lemmas~3.4 and~4.5]{PSl1}.

\begin{prop} \label{t-contramodules-relatively-cotorsion}
 Let $R$ be a commutative ring, $t\in R$ an element, $P$
a $t$\+contramodule $R$\+module, and $F$ a flat $R$\+module.
 Assume that the $R/tR$\+module $F/tF$ is projective.
 Then\/ $\Ext_R^i(F,P)=0$ for all $i\ge1$.
\end{prop}

 For the sake of completeness of the exposition in this paper, we
include a brief sketch of a proof of this crucial fact below.

 Denote by $K_t^\bullet=K_{R,t}^\bullet$ the two-term complex of
$R$\+modules $R\to R[t^{-1}]$, where the term $R$ is placed in
the cohomological degree~$-1$ and the term $R[t^{-1}]$ in
the cohomological degree~$0$.
 As above, we will use the simplified notation $\Ext_R^i(K_t^\bullet,M)=
\Hom_{\Der{\Modl R}}(K_t^\bullet,M[i])$ for any $R$\+module~$M$.

\begin{lem} \label{t-contramodules-described}
 For any $t$\+contramodule $R$\+module $P$, there exists
an $R$\+module $M$ such that $P$ is isomorphic to the $R$\+module\/
$\Ext^1_R(K^\bullet_t,M)$.
\end{lem}

\begin{proof}
 In fact, the functor $\Delta_t\colon M\longmapsto
\Ext^1_R(K^\bullet_t,M)$ is left adjoint to the embedding of
the full subcategory of $t$\+contramodule $R$\+modules into
the category of $R$\+modules~\cite[Theorem~6.4 and Remark~6.5]{Pcta}.
 More generally, for any multiplicative subset $S\subset R$ such that
the projective dimension of the $R$\+module $S^{-1}R$ does not
exceed~$1$, the functor $\Delta_{R,S}=\Ext_R^1(K_{R,S}^\bullet,{-})$
is left adjoint to the embedding of the full subcategory of
$S$\+contramodule $R$\+modules into $\Modl R$
\cite[Lemma~1.7(c) and Theorem~3.4(b)]{Pos}.

 For our purposes, it is sufficient to observe that the natural morphism
$P\to\Delta_t(P)$ is an isomorphism for any $t$\+contramodule $P$
by (the proof of) Lemma~\ref{S-contra-Delta}.
 So one can take $M=P$.
\end{proof}

 For any $R$\+module $M$, we consider two projective systems of
$R$\+modules, indexed by the positive integers.
 The first one is
$$
 M/tM\longleftarrow M/t^2M\longleftarrow M/t^3M\longleftarrow\dotsb,
$$
with the obvious surjective transition maps $M/t^nM\leftarrow
M/t^{n+1}M$ forming a commutative diagram with the natural surjections
$M\to M/t^nM$.
 The second projective system is
$$
 {}_tM\longleftarrow{}_{t^2}M\longleftarrow{}_{t^3}M\longleftarrow\dotsb,
$$
where for any element $s\in R$ we denote by ${}_sM\subset M$
the submodule of all elements annihilated by the action of~$s$ in~$M$.
 The transition map ${}_{t^n}M\leftarrow{}_{t^{n+1}}M$ acts by
the multiplication with~$t$.

\begin{lem} \label{ext-from-k-t-computed}
 For any $R$\+module $M$,
here is a natural short exact sequence of $R$\+modules
$$
 0\la\varprojlim\nolimits_n^1\,{}_{t^n}M\la
 \Ext^1_R(K_t^\bullet,M)\la\varprojlim\nolimits_n M/t^nM \la0.
$$
\end{lem}

\begin{proof}
 See~\cite[Lemma~6.7]{Pcta} or~\cite[Sublemma~4.6]{PSl1}.
\end{proof}

%

\begin{lem} \label{ext-into-lim-and-derived-lim-vanishing}
 Let $R$ be a commutative ring, $t\in R$ be an element, and $F$
a flat $R$\+module such that the $R/tR$\+module $F/tF$ is projective.
 Let $D_1\leftarrow D_2\leftarrow D_3\leftarrow\nobreak\dotsb$ be
a projective system of $R$\+modules such that $t^nD_n=0$ for all
$n\ge1$.
 Then the following hold true:
\begin{enumerate}
\item $\Ext_R^i(F,\,\varprojlim_n D_n)=0$ for all $i\ge1$; \par
\item $\Ext_R^i(F,\,\varprojlim_n^1 D_n)=0$ for all $i\ge1$.
\end{enumerate}
\end{lem}

\begin{proof}
%
%
%
%
 The argument for part~(1) is based on the dual
version~\cite[Proposition~18]{ET} of the Eklof Lemma~\cite[Lemma~1]{ET};
and part~(2) is deduced from part~(1) using the explicit construction
of $\varprojlim^1$.
 See~\cite[Lemma~9.9]{Pcta} or~\cite[Sublemma~4.7 and Lemma~3.4]{PSl1}
for the details.
\end{proof}

 Now we are prepared to conclude our argument.

\begin{proof}[Proof of
Proposition~\ref{t-contramodules-relatively-cotorsion}]
 According to Lemma~\ref{t-contramodules-described}, it suffices
to show that $\Ext^i_R(F,P)=0$ for $i\ge1$, where $F$ is a flat
$R$\+module such that the $R/tR$\+module $F/tF$ is projective and
$P=\Ext^1_R(K^\bullet_t,M)$ for some $R$\+module~$M$.
 This follows from Lemmas~\ref{ext-from-k-t-computed}
and~\ref{ext-into-lim-and-derived-lim-vanishing}~(1--2).
\end{proof}

\section{$S$-$\mathrm h$-nil rings} \label{S-h-nil-secn}


 The aim of this section is to show that, for any $S$\+h-nil ring $R$,
the projective dimension of the $R$\+module $R_S$ does not exceed~$1$.
 We will also obtain the direct product decomposition of
Proposition~\ref{product-decomp-exists}~(1) $\Rightarrow$ (2), (3)
without the bounded $S$\+torsion assumption on $R$, but assuming
that $R$ is $S$\+h-nil instead.

\begin{defn}
 We say that the ring $R$ is \emph{$S$\+h-nil} if every element $s\in S$
is contained only in finitely many maximal ideals of $R$ and every prime
ideal of $R$ intersecting $S$ is maximal.
\end{defn}

 Clearly, if $R$ is $S$\+h-nil, then it is $S$\+h-local.

\begin{lem} \label{L:S-h-nil}
 Let $S$ be a multiplicative subset of a commutative ring~$R$.
 Then the ring $R$ is $S$\+h-nil if and only if for any element $s\in S$
the ring $R/sR$ is a finite product of local rings and, in each of them,
the maximal ideal is nil (i.e., consists of nilpotent elements).
\end{lem}

\begin{proof}
 For any $S$\+h-local ring $R$ and an element $s\in S$, the ring $R/sR$
is a finite product of local rings, as it was observed in the proof of
Proposition~\ref{P:h-local}.
 It remains to recall that the nilradical of a commutative ring is equal to
the intersection of all its prime ideals, so a local ring has no nonmaximal
prime ideals if and only if its maximal ideal is nil.
\end{proof}


 Given an ideal $\aaa\subset R$, an $R$\+module $M$ is said to be
\emph{$\aaa$\+torsion} if it is $t$\+torsion (in the sense of
the definition in Section~\ref{t-contra-secn}) for every $t\in\aaa$.

\begin{lem} \label{S-torsion-m-torsion-local}
 Let $R$ be an $S$\+h-nil ring and\/ $\m$ a maximal ideal of $R$
intersecting~$S$.
 Then the following are equivalent for an $R$\+module $M$:
\begin{enumerate}
\item $M$ is an\/ $\m$\+torsion $R$\+module.
\item $M$ is an $S$\+torsion $R_\m$\+module.
\end{enumerate}
\end{lem}

\begin{proof}
 (1) $\Rightarrow$ (2) holds without the $S$\+h-nil hypothesis.
 If $M$ is $\m$\+torsion and $s\in\m\cap S$,
then $M$ is $s$\+torsion, hence also $S$\+torsion.
  Furthermore, for any commutative ring $R$ with a maximal ideal~$\m$,
any $\m$\+torsion $R$\+module is an $R_\m$\+module.

 (2) $\Rightarrow$ (1) Let $M$ be an $S$\+torsion $R_\m$\+module.
 Then for any element $x\in M$ there exists $s\in S$ such that $sx=0$
in~$M$.
 Now, if $s\notin\m$, then $x=0$; and if $s\in\m$, then $R_\m x
\subset M$ is a module over the ring $R_\m/sR_\m$.
 Since the maximal ideal of the local ring $R_\m/sR_\m$ is nil by
Lemma~\ref{L:S-h-nil}, any $(R_\m/sR_\m)$\+module is
$\m$\+torsion.
\end{proof}

\begin{lem} \label{S-torsion-m-torsion-direct-sum-decomp}
 Let $R$ be an $S$\+h-nil ring and $M$ an $R$\+module.
 Then the following are equivalent:
\begin{enumerate}
\item $M$ is $S$\+torsion.
\item There exists a collection of\/ $\m$\+torsion $R$\+modules
$M(\m)$, one for each\/ $\m\in\Max R$, \,$\m\cap S\ne\varnothing$,
such that $M$ is isomorphic to\/ $\bigoplus_\m M(\m)$.
\end{enumerate}
 The direct sum decomposition in~{\rm(2)} is unique if it exists.
\end{lem}

\begin{proof}
 (1) $\Rightarrow$ (2) Let $M$ be an $S$\+torsion $R$\+module.
 According to Proposition~\ref{P:h-local}, we have
$M\cong\bigoplus\limits_{\m\in\Max R;\,\m\cap S\ne\varnothing} M_\m$.
 For every~$\m$, the $R_\m$\+module $M(\m)=M_\m$ is also
$S$\+torsion, so it is $\m$\+torsion by
Lemma~\ref{S-torsion-m-torsion-local}.

 (2) $\Rightarrow$ (1) holds without the $S$\+h-nil hypothesis.
 It is sufficient to observe that, for any
$\m\in\Max R$, \,$\m\cap S\ne\varnothing$, any
$\m$\+torsion $R$\+module is $S$\+torsion
by Lemma~\ref{S-torsion-m-torsion-local}.

 The direct sum decomposition in~(2) is unique and functorial
(when it exists) by Lemma~\ref{sum-decomp-unique}, since any
$\m$\+torsion $R$\+module is an $S$\+torsion $R_\m$\+module.
 (This is valid for any $S$\+h-local ring~$R$.)
\end{proof}

 Let $\aaa\subset R$ be an ideal.
 An $R$\+module $P$ is said to be an \emph{$\aaa$\+contramodule} if
it is a $t$\+contramodule for every element $t\in\aaa$.
 According to Lemma~\ref{t-contra}~(3), it suffices to check this
condition for a set of generators of an ideal $\aaa$ (or even for a set
of generators of any ideal in $R$ whose radical coincides with
the radical of~$\aaa$).
 According to Lemma~\ref{t-contra}~(1), the full  subcategory of all
$\aaa$\+contramodule $R$\+modules is closed under the kernels,
cokernels, extensions, and infinite products in $\Modl R$.

\begin{lem} \label{a-contra-torsion-ext}
 Let\/ $\aaa$ be an ideal in a commutative ring~$R$.
 Then the following hold true:
\begin{enumerate}
\item If $N$ is an\/ $\aaa$\+torsion $R$\+module, then, for every
$R$\+module $X$, the $R$\+modules\/ $\Ext_R^i(N,X)$ are\/
$\aaa$\+contramodules for all $i\ge0$.
\item If $Q$ is an\/ $\aaa$\+contramodule $R$\+module, then,
for every $R$\+module $Y$, the $R$\+modules\/ $\Ext_R^i(Y,Q)$ are\/
$\aaa$\+contramodules for all $i\ge0$.
\end{enumerate}
\end{lem}

\begin{proof}
 This is a particular case of~\cite[Lemma~6.2(b)]{Pcta}.
\end{proof}

\begin{rem}
 A version of Lemma~\ref{a-contra-torsion-ext} with ``$\aaa$\+torsion
modules'' and ``$\aaa$\+contramodules'' replaced by ``$S$\+torsion
modules'' and ``$S$\+contramodules'', respectively, also holds for any
multiplicative subset $S$ in a commutative ring $R$ such that
$\operatorname{p.dim}_RR_S\le1$.
 This is provable by the same arguments as in~\cite[Lemma~6.2(b)]{Pcta}
(e.g., it can be deduced from Lemma~\ref{L:torsion-Hom} using
Lemma~\ref{S-contra}~(3--4)).
 However, our aim in this section is to \emph{prove} that
$\operatorname{p.dim}_RR_S\le1$ for $S$\+h-nil rings $R$, rather
than use this property.
 That is why we need $\aaa$\+contramodules here.
\end{rem}

\begin{lem} \label{S-contra-m-contra-local}
 Let $R$ be an $S$\+h-nil ring and\/ $\m$ a maximal ideal of $R$
intersecting~$S$.
 Then the following are equivalent for an $R$\+module~$P$:
\begin{enumerate}
\item $P$ is an\/ $\m$\+contramodule $R$\+module.
\item $P$ is an\/ $S$\+contramodule $R_\m$\+module.
\end{enumerate}
\end{lem}

\begin{proof}
 (1) $\Rightarrow$ (2) holds without the $S$\+h-nil hypothesis.
 If $P$ is an $\m$\+contramodule and $s\in\m\cap S$, then $P$ is
an $s$\+contramodule, hence also an $S$\+contramodule by
Lemma~\ref{S-contra}~(2).
 Furthermore, for any commutative ring $R$ with a maximal ideal~$\m$,
any $\m$\+contramodule $R$\+module is an $R_\m$\+module
by Lemma~\ref{t-contra}~(2).
 In fact, any $R$\+module in which the elements $1-t$ act by
automorphisms for all $t\in\m$ is an $R_\m$\+module
(cf.~\cite[Remark~9.2]{Pcta}).

 (2) $\Rightarrow$ (1) For simplicity of notation, we will denote
the image of $S$ in $R_\m$ also by~$S$.
 Consider the two-term complex of $R_\m$\+modules
$K^\bullet_\m\cong K^\bullet_{R_\m,S}$
$$
 R_\m\stackrel{\phi_\m}\la (R_\m)_S,
$$
and recall the notation $\Ext_{R_\m}^n(K^\bullet_\m,M)=
\Hom_{\Der{\Modl{R_\m}}}(K^\bullet_\m,M[n])$ and
$\Delta_{R_\m,S}(M)=\Ext_{R_\m}^1(K^\bullet_\m,M)$ from
Section~\ref{S-h-local-secn}.

 Let $P$ be an $S$\+contramodule $R_\m$\+module
(the terminology ``$S$\+contra\-module $R_\m$\+module'' is
unambiguous by Lemma~\ref{S-contra}~(1)).
 According to Lemma~\ref{S-contra-Delta}, there exists
an $R_\m$\+module $M$ such that $P\cong\Delta_{R_\m,S}(M)$.
 The cohomology modules $H^{-1}(K^\bullet_\m)=I_\m$ and
$H^0(K^\bullet_\m)=(R_\m)_S/\phi_\m(R_\m)$ of the complex
$K^\bullet_\m$ are $S$\+torsion $R_\m$\+modules, so they are also
$R_\m\m$\+torsion $R_\m$\+modules by
Lemma~\ref{S-torsion-m-torsion-local}.
 We have a distinguished triangle in the derived category
$\Der{\Modl{R_\m}}$
$$
 H^{-1}(K^\bullet_\m)[1]\la K^\bullet_\m\la
 H^0(K^\bullet_\m)\la H^{-1}(K^\bullet_\m)[2],
$$
hence the induced long exact sequence of modules
$\Hom_{\Der{\Modl{R_\m}}}({-},M[*])$
\begin{multline*}
 0\la\Ext^1_{R_\m}(H^0(K^\bullet_\m),M)\la
 \Ext^1_{R_\m}(K^\bullet_\m,M) \\ \la
 \Hom_{R_\m}(H^{-1}(K^\bullet_\m),M)\la
 \Ext^2_{R_\m}(H^0(K^\bullet_\m),M)\la\dotsb
\end{multline*}
 Now the $R_\m$\+modules $\Ext^n_{R_\m}(H^i(K^\bullet_\m),M)$ are
$R_\m\m$\+contramodules for all $n\ge0$ and $i=-1$, $0$
by Lemma~\ref{a-contra-torsion-ext}~(1), and the full subcategory
of $R_\m\m$\+contramodule $R_\m$\+modules is closed under
the kernels, cokernels, and extensions in $\Modl{R_\m}$ by
Lemma~\ref{t-contra}~(1).
 Thus $P\cong\Delta_{R_\m,S}(M)=\Ext^1_{R_\m}(K^\bullet_\m,M)$
is also an $R_\m\m$\+contramodule $R_\m$\+module, hence
an $\m$\+contramodule $R$\+module.
\end{proof}

\begin{lem} \label{ext-between-products}
 Let $R$ be an $S$\+h-nil ring.
 Then the following hold true:
\begin{enumerate}
\item Let $M(\m)$ and $N(\m)$ be two collections of $R$\+modules
indexed by\/ $\m\in\Max R$, \,$\m\cap S\ne\varnothing$, such that
$M(\m)$ is\/ $\m$\+torsion and $N(\m)$ is an $R_\m$\+module for
every\/~$\m$.
 Then, for any $i\ge0$, any extension class in\/
$\Ext_R^i\bigl(\bigoplus_\m M(\m),\,\bigoplus_\m N(\m)\bigr)$ is
a direct sum over\/~$\m$ of some extension classes in\/
$\Ext_R^i(M(\m),N(\m))$, that is
$$
 \textstyle\Ext_R^i\bigl(\bigoplus_\m M(\m),\,\bigoplus_\m N(\m)\bigr)
 \,\cong\,\displaystyle\prod_\m\textstyle\Ext_R^i(M(\m),N(\m)).
$$
\item Let $P(\m)$ and $Q(\m)$ be two collections of $R$\+modules
indexed by\/ $\m\in\Max R$, \,$\m\cap S\ne\varnothing$, such that
$P(\m)$ is an $R_\m$\+module and $Q(\m)$ is an $\m$\+contramodule
for every\/~$\m$.
  Then, for any $i\ge0$, any extension class in\/
$\Ext_R^i\bigl(\prod_\m P(\m),\,\prod_\m Q(\m)\bigr)$ is
a product over\/~$\m$ of some extension classes in\/
$\Ext_R^i(P(\m),Q(\m))$, that is
$$
 \textstyle\Ext_R^i\bigl(\prod_\m P(\m),\,\prod_\m Q(\m)\bigr)
 \,\cong\,\displaystyle\prod_\m\textstyle\Ext_R^i(P(\m),Q(\m)).
$$
\end{enumerate}
\end{lem}

\begin{proof}
 (1) Notice that the case $i=0$ is covered by
Lemma~\ref{sum-decomp-unique}.
 The following argument is applicable for all $i\ge0$.

 It suffices to show that
$\Ext_R^i\bigl(M(\m),\bigoplus_{\n\ne\m}N(\n)\bigr)=0$,
where the direct sum of $R_\n$\+modules $N(\n)$ runs over
the maximal ideals~$\n$ of $R$ different from~$\m$, and $\m$~is
a maximal ideal intersecting~$S$.
 Indeed, let $s$ be an element of the intersection $\m\cap S$.
 Then the action of~$s$ in the direct sum $\bigoplus_{s\notin\n}N(\n)$
is invertible, while $M(\m)$ is $s$\+torsion.

 For any element $t\in R$, and any $R$\+modules $M$, $N$ such that
the action of~$t$ is invertible in $N$, while $M$ is $t$\+torsion,
we have $\Ext^i_R(M,N)=\Ext^i_{R[t^{-1}]}(R[t^{-1}]\otimes_R M,\>N)=0$.
 Hence $\Ext^i_R\bigl(M(\m),\,\bigoplus_{s\notin\n}N(\n)\bigr)=0$.

 The set of all maximal ideals~$\n$ of $R$ containing~$s$ is finite, so
it remains to show that $\Ext_R^i(M(\m),N(\n))=0$ for any fixed two
maximal ideals $\m\ne\n$ of~$R$.
 Indeed, let $r\in\m\setminus\n$ be an element of the complement.
 Then the action of~$r$ is invertible in $N(\n)$, while $M(\m)$ is
$r$\+torsion.
 For the same reason as above, the Ext module in question vanishes.

 (2) The case $i=0$ is covered by Lemma~\ref{product-decomp-unique}.
 The following argument is applicable for all $i\ge0$.

 It suffices to show that
$\Ext_R^i\bigl(\prod_{\m\ne\n}P(\m),\,Q(\n)\bigr)=0$,
where the product of $R_\m$\+modules $P(\m)$ runs over
the maximal ideals~$\m$ of $R$ different from~$\n$, and
$\n$~is a maximal ideal intersecting~$S$.
 Choose an element $s\in\n\cap S$.
 Then the action of~$s$ in the product $\prod_{s\notin\m}P(\m)$
is invertible, while $Q(\n)$ is an $s$\+contramodule.

 For any element $t\in R$, and any $R$\+modules $P$, $Q$ such that
the action of~$t$ is invertible in $P$, while $Q$ is a $t$\+contramodule,
we have $\Ext_R^i(P,Q)=0$.
 Indeed, the $R$\+module $E=\Ext_R^i(P,Q)$ is a $t$\+contramodule
by Lemma~\ref{a-contra-torsion-ext}~(2), and at the same time $t$~acts
invertibly in $E$, so it follows that $E=0$.

 Hence, in particular,
$\Ext_R^i\bigl(\prod_{s\notin\m}P(\m),\,Q(\n)\bigr)=0$.
 The set of all maximal ideals~$\m$ of $R$ containing~$s$ is finite,
so it remains to show that $\Ext_R^i(P(\m),Q(\n))=0$ for any fixed
two maximal ideals $\m\ne\n$.
 Let $r\in\n\setminus\m$ be an element of the complement.
 Then the action of~$r$ is invertible in $P(\m)$, while $Q(\n)$ is
an $r$\+contramodule.
 Therefore, the Ext module in question vanishes.
\end{proof}

\begin{cor} \label{product-decomp-closed-under}
 Let $R$ be an $S$\+h-nil ring.
 Then the full subcategory of all $R$\+modules of the form\/
$\prod\limits_{\m\in\Max R;\,\m\cap S\ne\varnothing}P(\m)$, where
$P(\m)$ are\/ $\m$\+contramodule $R$\+modules, is closed under
the kernels, cokernels, extensions, and infinite products in $\Modl R$.
\end{cor}

\begin{proof}
 Follows from Lemma~\ref{ext-between-products}~(2) for $i=0$ and~$1$.
\end{proof}

\begin{lem} \label{ext-from-S-torsion-decomp}
 Let $R$ be an $S$\+h-nil ring.
 Then for any $S$\+torsion $R$\+module $N$, any $R$\+module $X$,
and any $n\ge0$, there exists a collection of\/ $\m$\+contramodule
$R$\+modules $P(\m)$, one for each\/ $\m\in\Max R$,
\,$\m\cap S\ne\varnothing$, such that the $R$\+module\/ $\Ext^n_R(N,X)$
is isomorphic to\/ $\prod_\m P(\m)$.
\end{lem}

\begin{proof}
 By Lemma~\ref{S-torsion-m-torsion-direct-sum-decomp}, we have
$N\cong\bigoplus_\m N(\m)$, where $N(\m)$ are some $\m$\+torsion
$R$\+modules.
 It remains to set $P(\m)=\Ext^n_R(N(\m),X)$ and apply
Lemma~\ref{a-contra-torsion-ext}~(1). 
\end{proof}

\begin{lem} \label{ext-from-K-decomp}
 Let $R$ be an $S$\+h-nil ring and $K^\bullet$ the two-term complex
$R\stackrel\phi\to R_S$.
 Then, for any $R$\+module $M$ and any $n\ge0$, there exists a collection
of\/ $\m$\+contramodule $R$\+modules $P(\m)$, one for each\/
$\m\in\Max R$, \,$\m\cap S\ne\varnothing$, such that the $R$\+module\/
$\Ext^n_R(K^\bullet,M)$ is isomorphic to\/ $\prod_\m P(\m)$.
\end{lem}

\begin{proof}
 The cohomology modules $H^{-1}(K^\bullet)=I$ and
$H^0(K^\bullet)=R_S/\phi(R)$ of the complex $K^\bullet$ are
$S$\+torsion, since the complex $R_S\otimes_R K^\bullet$ is acyclic.
 We have a distinguished triangle in the derived category $\Der{\Modl R}$
$$
 H^{-1}(K^\bullet)[1]\la K^\bullet\la H^0(K^\bullet)\la
 H^{-1}(K^\bullet)[2],
$$
hence the induced long exact sequence of modules
$\Hom_{\Der{\Modl R}}({-},M[*])$
\begin{multline*}
 \dotsb\la\Ext^{n-2}(H^{-1}(K^\bullet),M)
 \la\Ext^n_R(H^0(K^\bullet),M) \\
 \la\Ext^n_R(K^\bullet,M)\la\\
 \Ext_R^{n-1}(H^{-1}(K^\bullet),M)\la
 \Ext^{n+1}(H^0(K^\bullet),M)\la\dotsb 
\end{multline*}
 Now the $R$\+modules $\Ext_R^n(H^i(K^\bullet),M)$ are products
of $\m$\+contramodule $R$\+modules over $\m\in\Max R$,
\,$\m\cap S\ne\varnothing$ for all $n\ge0$ and $i=0$, $-1$
by Lemma~\ref{ext-from-S-torsion-decomp}; and so are all the extensions of
the kernels and cokernels of morphisms between such $R$\+modules, by
Corollary~\ref{product-decomp-closed-under}.
\end{proof}

\begin{rem}
 It follows from Lemma~\ref{ext-between-products}~(1) for $i=2$
together with Lemma~\ref{S-torsion-m-torsion-direct-sum-decomp} that,
for any $S$\+h-nil ring $R$, the complex $K^\bullet=K^\bullet_{R,S}$,
as an object of the derived category $\Der{\Modl R}$, is isomorphic to
the direct sum
$\bigoplus\limits_{\m\in\Max R;\,\m\cap S\ne\varnothing}K^\bullet_\m$.
 In fact, for any two-term complex of $R$\+modules $L^\bullet$ with
$S$\+torsion cohomology modules, there is a natural isomorphism
$L^\bullet\cong\bigoplus\limits_{\m\in\Max R;\,\m\cap S\ne\varnothing}
L^\bullet_\m$ in $\Der{\Modl R}$.
 This observation provides another way to prove
Lemma~\ref{ext-from-K-decomp}.
\end{rem}

 Having finished all the preparatory work, we can now deduce the main
results of this section.

\begin{thm} \label{T:S-h-nil-is-Matlis}
 For any $S$\+h-nil commutative ring $R$, the projective dimension
of the $R$\+module $R_S$ does not exceed\/~$1$.
\end{thm}

\begin{proof}
 Let $M$ be an $R$\+module and $n\ge2$ be an integer.
 By Lemma~\ref{ext-from-K}~(3), we have an isomorphism
$\Ext_R^n(K^\bullet,M)\cong\Ext_R^n(R_S,M)$.
 By Lemma~\ref{ext-from-K-decomp}, the $R$\+module
$\Ext^n_R(K^\bullet,M)$ is isomorphic to an infinite product of
the form $\prod\limits_{\m\in\Max R;\,\m\cap S\ne\varnothing}P(\m)$
for some $\m$\+contramodule $R$\+modules $P(\m)$.
 On the other hand, $\Ext_R^n(R_S,M)$ is an $R_S$\+module.

 It follows that $P(\m)$, being a direct summand of
$\Ext_R^n(K^\bullet,M)$, is an $R_S$\+module for every~$\m$.
 So, if $s$~is an element of the intersection $\m\cap S$, then $P(\m)$
is simultaneously an $s$\+contramodule $R$\+module and
an $R[s^{-1}]$\+module.
 Thus $P(\m)=0$ and $\Ext_R^n(R_S,M)=0$.
\end{proof}

\begin{cor} \label{S-h-nil-S-h-divisible}
 Let $R$ be an $S$\+h-nil ring.
 Then every $S$\+divisible $R$\+module is $S$\+h-divisible.
\end{cor}

\begin{proof}
 This follows immediately from Theorem~\ref{T:S-h-nil-is-Matlis}
and Proposition~\ref{P:Matlis}.
 Here is an alternative argument, based on the techniques developed
in this and the previous section.
 For any $R$\+module $M$, we have the exact sequence of
$R$\+modules~($\ast\ast$) from Section~\ref{S-h-local-secn}:
\begin{multline*}
 0\la\Ext^0_R(K^\bullet,M)\la\Hom_R(R_S,M) \la M \\
 \la \Ext^1_R(K^\bullet,M)\la\Ext^1_R(R_S,M)\la0.
\end{multline*}
 By Lemma~\ref{ext-from-K-decomp}, the $R$\+module
$\Ext_R^1(K^\bullet,M)$ can be presented as a product of some
$\m$\+contramodule $R$\+modules $P(\m)$ over all
$\m\in\Max R$, \,$\m\cap S\ne\varnothing$.

 Now assume that $M$ is $S$\+divisible.
 Let $s\in\m\cap S$ be an element of the intersection.
 Then $M$ is an $s$\+divisible module, while $\Hom_R(R[s^{-1}],P(\m))=0$
implies that $P(\m)$ has no $s$\+divisible submodules.
 Thus any morphism $M\to P(\m)$ vanishes, the morphism
$M\to\Ext^1_R(K^\bullet,M)$ vanishes, and the map
$\Hom_R(R_S,M)\to M$ is surjective.  
\end{proof}

\begin{cor} \label{S-contra-m-contra-product-decomp}
 Let $R$ be an $S$\+h-nil ring and $P$ an $R$\+module.
 Then the following are equivalent:
\begin{enumerate}
\item $P$ is an $S$\+contramodule.
\item There exists a collection of\/ $\m$\+contramodule $R$\+modules
$P(\m)$, one for each\/ $\m\in\Max R$, \,$\m\cap S\ne\varnothing$,
such that $P$ is isomorphic to\/ $\prod_\m P(\m)$.
\end{enumerate}
 The direct product decomposition in~{\rm(2)} is unique if it exists.
\end{cor}

\begin{proof}
 (1) $\Rightarrow$ (2) Let $P$ be an $S$\+contramodule $R$\+module.
 According to Lemma~\ref{S-contra-Delta}, there exists an $R$\+module
$M$ such that $P\cong\Delta(M)=\Ext_R^1(K^\bullet,M)$.
 The desired direct product decomposition is now provided by
Lemma~\ref{ext-from-K-decomp}.

 (2) $\Rightarrow$ (1) holds without the $S$\+h-nil hypothesis.
 It suffices to observe that, for any $\m\in\Max R$,
\,$\m\cap S=\varnothing$, any $\m$\+contramodule $R$\+module
is an $S$\+contramodule by Lemma~\ref{S-contra-m-contra-local}.

 The direct product decomposition in~(2) is unique and functorial (when it
exists) by Lemma~\ref{product-decomp-unique}, since any
$\m$\+contramodule $R$\+module is an $S$\+contramodule
$R_\m$\+module.
 (This is valid for any $S$\+h-local ring~$R$.)
\end{proof}

\begin{cor} \label{S-h-nil-product-decomp-exists}
 Let $R$ be an $S$\+h-nil ring.
 Then every $S$\+contramodule $P$ is isomorphic to the direct product
of its colocalizations\/
$\prod\limits_{\m\in\Max R;\,\m\cap S\ne\varnothing} P^\m$.
 Furthermore, one has $P^\n=0$ for any\/ $\n\in\Max R$ such that\/
$\n\cap S=\varnothing$.
\end{cor}

\begin{proof}
 According to Corollary~\ref{S-contra-m-contra-product-decomp},
there exists a collection of $\m$\+contra\-module $R$\+modules $P(\m)$,
one for each $\m\in\Max R$, \,$\m\cap S\ne\varnothing$, such that
$P\cong\prod_\m P(\m)$.
 Now we have $P(\m)^\m\simeq P(\m)$, since $P(\m)$ is an $R_\m$\+module.
 For any $\n\in\Max R$, \,$\n\ne\m$, the $R$\+module $P(\m)^\n$ is
an $(R_\m\otimes R_\n)$-module, so by
Lemma~\ref{two-maximal-ideals} it is an $R_S$\+module.
 On the other hand, $P(\m)$ is an $S$\+contramodule by
Lemma~\ref{S-contra-m-contra-local}, hence $P(\m)^\n$ is
an $S$\+contramodule by Lemma~\ref{L:torsion-Hom}~(2).
 Thus $P(\m)^\n=0$.
 It follows that $P^\m\cong P(\m)$ for all $\m\in\Max R$,
\,$\m\cap S\ne\varnothing$, and $P^\n=0$ for all $\n\in\Max R$,
\,$\n\cap S=\varnothing$.
\end{proof}


\begin{cor} Let $R$ be an $S$\+h-nil ring. The following hold:
 \begin{enumerate}
 \item The category of $S$\+contramodule
$R$\+modules is abelian.
\item The mutually inverse functors $P\mapsto (P^\m)_\m$ and
$(P(\m))_\m\mapsto\prod_\m P(\m)$ establish an equivalence between
the abelian category of $S$\+contramodule $R$\+modules and
the Cartesian product of the abelian categories of $\m$\+contramodule
$R$\+modules over all $\m\in\Max R$, \,$\m\cap S\ne\varnothing$.
\end{enumerate}
\end{cor}
\begin{proof}
(1) follows by Theorem~\ref{T:S-h-nil-is-Matlis} and
Lemma~\ref{S-contra}~(3--4).

(2) follows from Corollaries~\ref{S-contra-m-contra-product-decomp}
and~\ref{S-h-nil-product-decomp-exists}.
\end{proof}

\section{$S$-almost perfect rings} 
\label{S-almost-perfect-secn}


 The main result of this section is
Theorem~\ref{T:characterization}, which provides
a characterization of $S$\+almost perfect rings in terms of
ring-theoretic and homological properties.
 We also describe $S$\+strongly flat modules over
an $S$\+h-nil ring $R$ in
Proposition~\ref{P:S-strongly-flat-over-S-h-nil}.

\begin{lem}\label{L:perfect-Matlis}
 If $R/sR$ is a perfect ring for every $s\in S$, then the ring $R$ is
$S$\+h-nil.
 Hence, in particular, $R$ is $S$\+h-local and the $R$\+module $R_S$
has projective dimension at most one.
\end{lem}

\begin{proof}
 The fact that every element $s\in S$ is contained in only finitely many
maximal ideals follows by Lemma~\ref{L:perfect}~(3).
 Let $\p$ be a prime ideal containing an element~$s$ of $S$.
 Then, $R/\p$ is a perfect ring, being a quotient of $R/sR$, and it is
a field since it is a domain; thus $\p$~is maximal.
 
 By Theorem~\ref{T:S-h-nil-is-Matlis}, we have
$\operatorname{p.dim}_RR_S\leq1$.
\end{proof}

\begin{lem}\label{L:1}
 Assume that\/ $\operatorname{p.dim}_RR_S\leq 1$.
 Let $M$ be an $R$\+module such that $M\otimes_RR_S$ is a projective
$R_S$\+module.
 Then the $R$\+module $M$ is $S$\+strongly flat if and only if\/
$\Ext^1_R(M,P)=0=\Ext^2_R(M,P)$ for every $S$\+contramodule
(i.e., $S$\+reduced $S$\+weakly cotorsion) $R$\+module $P$.
\end{lem}

\begin{proof}
 When $\operatorname{p.dim}_RR_S\le1$, the projective dimension of
any $S$\+strongly flat $R$\+module also does not exceed~$1$.
 So the condition is clearly necessary.

 Let $C$ be an $S$\+weakly cotorsion $R$\+module, and let
$h(C)$ be its maximal $S$\+h-divisible submodule.
 We have two short exact sequences
(cf.\ the long exact sequence~($\ast\ast$) in
Section~\ref{S-h-local-secn}):
\begin{gather}  \tag{1}
 0\la\Hom_R(R_S/\phi(R),C)\la\Hom_R(R_S,C)\la h(C)\la 0 \\
 0\la h(C)\la C\la C/h(C)\la 0 \tag{2}
\end{gather}

 Furthermore, since $\operatorname{p.dim}_RR_S\le1$,
the class of $S$\+weakly cotorsion $R$\+modules is closed under
epimorphic images.
 So $C/h(C)$ is $S$\+weakly cotorsion.
 The $R$\+module $C/h(C)$ is $S$\+h-reduced by
Proposition~\ref{P:Matlis}~(2) or by~\cite[proof of Lemma~1.8(a)]{Pos}.
 Thus $C/h(C)$ is an $S$\+contramodule.
 Alternatively, one can see from the long exact sequence~($\ast\ast$)
that $C/h(C)\cong\Delta(C)$, since $C$ is $S$\+weakly cotorsion.
 Since $\operatorname{p.dim}_RR_S\le1$, the $R$\+module
$\Delta(C)$ is an $S$\+contramodule by~\cite[Lemma~1.7(c)]{Pos}. 

 The $R$\+module $R_S/\phi(R)$ is $S$\+torsion, hence
the $R$\+module $A=\Hom_R(R_S/\phi(R),C)$ is an $S$\+contramodule
by Lemma~\ref{L:torsion-Hom}~(1).
 
 Now $\Ext^1_R(M,\Hom_R(R_S,C))=0$ since $M\otimes_RR_S$ is
a projective $R_S$\+module, and $\Ext^1_R(M,C/h(C))=0=\Ext^2_R(M,A)$
since $C/h(C)$ and $A$ are $S$\+contramodules.
 From the exact sequence~(1) we obtain
$0=\Ext^1_R(M,\Hom_R(R_S,C))\to\Ext^1_R(M,h(C))\to\Ext^2_R(M,A)=0$,
and from the exact sequence~(2), 
\,$0=\Ext^1_R(M,h(C))\to\Ext^1_R(M,C)\to\Ext^1_R(M,C/h(C))=0$.
\end{proof}
 
\begin{cor}\label{C:1}
 Let $S$ be a multiplicative subset of a commutative ring $R$ such
that $R_S$ is a perfect ring and\/ $\operatorname{p.dim}_RR_S\leq 1$.
 Then a flat $R$\+module $F$ is $S$\+strongly flat if and only if\/
$\Ext^1_R(F,P)=0=\Ext^2_R(F,P)$ for every $S$\+contramodule
$R$\+module~$P$.
\end{cor}

\begin{proof}
 Follows from Lemma~\ref{L:1}.
\end{proof}

\begin{lem}\label{s-contra-enochs-cotorsion}
 Let $R$ be a commutative ring and $s\in R$ an element such that
the ring $R/sR$ is perfect.
 Then every $s$\+contramodule $R$\+module is Enochs cotorsion.
\end{lem}
\begin{proof}
 Follows from Proposition~\ref{t-contramodules-relatively-cotorsion}.
\end{proof}

\begin{cor} \label{S-contramod-Enochs-cotorsion}
 Let $S$ be a multiplicative subset of a commutative ring $R$ such that
the ring $R/sR$ is perfect for every $s\in S$.
 Then every $S$\+contramodule $R$\+module is Enochs cotorsion.
\end{cor}

\begin{proof}
 The ring $R$ is $S$\+h-nil by Lemma~\ref{L:perfect-Matlis}, so
Corollary~\ref{S-contra-m-contra-product-decomp} applies.
 Thus every $S$\+contramodule $R$\+module $P$ is isomorphic to
a direct product of the form
$\prod\limits_{\m\in\Max R;\,\m\cap S\ne\varnothing} P(\m)$, where
$P(\m)$ are some $\m$\+contramodule $R$\+modules.
 Let $s$~be an element of the intersection $\m\cap S$; then $P(\m)$ is
an $s$\+contramodule.
 According to Lemma~\ref{s-contra-enochs-cotorsion}, $P(\m)$ is
an Enochs cotorsion $R$\+module, and it follows that $P$ is
Enochs cotorsion, too.
\end{proof}

\begin{defn} If $S$ is a multiplicative subset of a commutative ring $R$,
we say that $R$ is \emph{$S$\+almost perfect} if $R_S$ is a perfect ring
and $R/sR$ is a perfect ring for every $s\in S$.

 We say that $R$ is \emph{$S$\+semiartinian} if every nonzero quotient
of $R$ modulo an ideal intersecting $S$ contains a simple module.
\end{defn}

 Now we can deduce the following proposition, which is one of our
main results.

\begin{prop} \label{S-almost-perfect-implies}
 Let $R$ be an $S$\+almost perfect commutative ring.
 Then all flat $R$\+modules are $S$\+strongly flat.
\end{prop}

\begin{proof}
 By Lemma~\ref{L:perfect-Matlis}, we have
$\operatorname{p.dim}_RR_S\le1$.
 Thus Corollary~\ref{C:1} applies, and it remains to show that
$\Ext_R^1(F,P)=0=\Ext_R^2(F,P)$ for all flat $R$\+modules $F$
and $S$\+contramodule $R$\+modules~$P$.
 This follows immediately from
Corollary~\ref{S-contramod-Enochs-cotorsion}.
\end{proof}

 The next lemma provides the converse implication.

\begin{lem} \label{S-almost-perfect-implied-by}
 Let $S$ be a multiplicative subset of a commutative ring~$R$.
 Assume that every flat $R$\+module is $S$\+strongly flat.
 Then the ring $R$ is $S$\+almost perfect.
\end{lem}

\begin{proof}
 One can easily deduce this from Lemmas~\ref{L:R_S-modules}
and~\ref{L:R/sR-modules}, but we prefer to give a direct argument
as well.

 Let $G$ be a flat $R_S$\+module; then $G$ is also a flat $R$\+module.
 By assumption, $G$ is an $S$\+strongly flat $R$\+module.
 By Lemma~\ref{L:S-strongly}~(1), it follows that
$G\cong G\otimes_R R_S$ is a projective $R_S$\+module.
 We have shown that all flat $R_S$\+modules are projective;
so $R_S$ is a perfect ring.

 Let $s$~be an element of~$S$.
 In order to show that $R/sR$ is a perfect ring, we will check that
every Bass flat $R/sR$\+module is projective.
 Let $\overline a_1$, $\overline a_2$,~\dots, $\overline a_n$,~\dots\ be
a sequence of elements of $R/sR$.
 The \emph{Bass flat $R/sR$\+module} $\overline B$ is the direct limit
of the sequence
$$
 R/sR\stackrel{\overline a_1}\la R/sR\stackrel{\overline a_2}\la
 R/sR\stackrel{\overline a_3}\la\dotsb
$$
 Let $a_n\in R$ be some preimages of the elements
$\overline a_n\in R/sR$.
 Denote by $B$ the direct limit of the sequence of $R$\+modules
$R\stackrel{a_1}\to R\stackrel{a_2}\to R\stackrel{a_3}\to\dotsb$.
 Then $B$ is a (Bass) flat $R$\+module and $\overline B\cong B/sB$.

 By assumption, $B$ is an $S$\+strongly flat $R$\+module.
 By Lemma~\ref{L:S-strongly}~(2), it follows that $\overline B$ is
a projective $R/sR$\+module.
\end{proof}

 We now give the following characterization of $S$\+almost perfect rings.

\begin{thm}\label{T:characterization}
 Let $R$ be a commutative ring and $S$ a multiplicative subset of~$R$.
 Then the following are equivalent:
\begin{enumerate}
\item $R$ is $S$\+almost perfect. 
\item Every flat $R$\+module is $S$\+strongly flat.
\item Every $S$\+weakly cotorsion $R$\+module is Enochs cotorsion.
\item $R$ is $S$\+h-local and $S$\+semiartian, and $R_S$ is perfect.
\item The class of $S$\+strongly flat $R$\+modules is a covering class.
\end{enumerate}
\end{thm}

\begin{proof} 
(1) $\Rightarrow$ (2) is Proposition~\ref{S-almost-perfect-implies};
(2) $\Rightarrow$ (1) is Lemma~\ref{S-almost-perfect-implied-by}.

(2) $\Leftrightarrow$ (3) Follows immediately by the definitions.

(1) $\Leftrightarrow$ (4) Follows from the definitions,
Lemmas~\ref{L:perfect-Matlis} and~\ref{L:perfect}~(4).

(2) $\Rightarrow$ (5) Clear, since the class of flat modules is a covering class.

(5) $\Rightarrow$ (1) Follows by Lemmas~\ref{L:R_S-modules}
and~\ref{L:R/sR-modules}.
\end{proof}

\begin{expl}
 Let $R$ be a Noetherian commutative ring of Krull dimension not
exceeding~$1$.
 Choose a set $P_0$ of minimal prime ideals of $R$ such that all
the prime ideals of $R$ not belonging to $P_0$ are maximal.
 For example, one take $P_0$ to be the set of all minimal prime
ideals in~$R$ \cite[Section~13]{Pcta}.
 Put $S=R\setminus\bigcup\limits_{\p\in P_0}\p$.
 Then the ring $R_S$ is a Noetherian commutative ring of Krull
dimension~$0$, and so is the ring $R/sR$ for all $s\in S$.
 Hence these are Artinian rings, and consequently, perfect rings.
 Therefore, the ring $R$ is $S$\+almost perfect.

 Notice that the classical ring of quotients of the ring $R$ fails
to be Artinian for some Noetherian commutative rings $R$ of
Krull dimension~$1$ containing nilpotent elements.
 In this case, the ring $R$ is \emph{not} almost perfect in the sense
of the paper~\cite{FS}.
 Still, it is $S$\+almost perfect for the above-described multiplicative
subset $S\subset R$ containing some zero-divisors.

 For example, this happens for the ring $R=k[x,y]/(xy,y^2)$,
where $k$~is a field.
 The multiplicative subset $S$ contains the zero-divisor $x\in R$
in this case.
\end{expl}

\begin{rem}
 Conversely, let $R$ be an $S$\+almost perfect ring and
$\p\varsubsetneq\q$ be two prime ideals of~$R$.
 Then the ring $R_S$, being a perfect ring, has Krull dimension~$0$.
 Hence the ideal~$\q$ must intersect~$S$.
 Let $s\in\q\cap S$ be an element of the intersection.
 Since the ring $R/sR$ is also perfect, and therefore has Krull
dimension~$0$, the element~$s$ cannot belong to~$\p$.
 Thus we have $\p\cap S=\varnothing$.
 It follows that there \emph{cannot} be three embedded prime ideals
$\p\varsubsetneq\q\varsubsetneq\m$ in $R$, so the Krull dimension
of $R$ does not exceed~$1$.
\end{rem}

\begin{expl}
 Let $K$ be a field, and let $\bbN$ denote the set of all positive
integers.
 Consider the ring $K^\bbN$ (the product of $\bbN$ copies of~$K$) and
denote by $\ep$ its unit element.
 Consider the subring
$$
 R = K^{(\bbN)}+\ep K\subset K^\bbN.
$$

 Then $R$ consists of all the elements of the direct product $K^\bbN$
of the form
$$
 R=\{(x_1,x_2,\dotsc,x_n,x,x,x,\dotsc)\mid \exists n\in\bbN;
 \ x_i,x\in K\}.
$$
 The ring $R$ is a Von Neumann regular commutative ring, so every
localization of $R$ at a maximal (\,$=$~prime) ideal is a field.
 In particular, the Krull dimension of $R$ is equal to~$0$.

 The ideal $\m_0=K^{(\bbN)}\subset R$ is a maximal ideal of $R$ and
$R/\m_0\cong K$ as shown by the epimorphism $R\to K$;
$(x_1,x_2,\dotsc,x_n,x,x,x,\dotsc)\mapsto x$.
 The other maximal ideals of $R$ are the ideals $\m_i=R\cap
\prod\limits_{j\in\bbN\setminus\{i\}}K$ for every $i\in\bbN$.

 Set $S=R\setminus\m_0$.
 Then $R_S\cong K$ is a field.
 For every element $s\in S$, there is $n\in\bbN$ such that
$s=(x_1,x_2,\dotsc,x_n,x,x,x,\dotsc)$ with $x\ne0$.
 Hence $s$~belongs, at most, to the $n$~maximal ideals
$\m_i$, $i=1$,~\dots,~$n$.

 So the ring $R/sR$ is semilocal of Krull dimension~$0$ and
its localizations at its maximal ideals are fields.
 In fact, $R/sR$ is the direct product of a finite number of copies
of~$K$.
 We conclude that $R$ is an $S$\+almost perfect ring.
\end{expl}

 The next proposition generalizes the equivalence
(1) $\Leftrightarrow$ (2) in Theorem~\ref{T:characterization}.
 It also means that~\cite[Optimistic Conjecture~1.1]{PSl2} holds for
$S$\+h-nil rings (cf.~\cite[Theorems~1.3--1.5]{PSl2}).

\begin{prop} \label{P:S-strongly-flat-over-S-h-nil}
 Let $R$ be an $S$\+h-nil ring and $F$ be a flat $R$\+module.
 Then the following are equivalent:
\begin{enumerate}
\item The $R$\+module $F$ is $S$\+strongly flat.
\item The $R_S$\+module $F\otimes_RR_S$ is projective, and
      the $R/sR$\+module $F/sF$ is projective for every $s\in S$.
\end{enumerate}
\end{prop}

\begin{proof}
 (1) $\Rightarrow$ (2) Follows by Lemma~\ref{L:S-strongly}~(1--2).

 (2) $\Rightarrow$ (1) By Theorem~\ref{T:S-h-nil-is-Matlis}, we have
$\operatorname{p.dim}_RR_S\le1$.
 So Lemma~\ref{L:1} applies, and it suffices to show that
$\Ext_R^1(F,P)=0=\Ext_R^2(F,P)$ for any $S$\+contramodule
$R$\+module~$P$.
 By Corollary~\ref{S-contra-m-contra-product-decomp}, $P$ can be
decomposed into a direct product,
$P\cong\prod\limits_{\m\in\Max R;\,\m\cap S\ne\varnothing} P(\m)$,
where $P(\m)$ are $\m$\+contramodule $R$\+modules.
 Let $s$ be an element of the intersection $\m\cap S$; then $P(\m)$
is an $s$\+contramodule.
 By Proposition~\ref{t-contramodules-relatively-cotorsion}, we have
$\Ext_R^i(F,P(\m))=0$ for all $i\ge1$.
 Hence also $\Ext_R^i(F,P)=0$.
\end{proof}

%
%
%
%
%

\section{The condition $\clP_1=\F_1$} \label{P1=F1-secn}


 In this section we prove that if $R$ is an $S$\+almost perfect ring then every module with flat dimension $\leq 1$ has projective dimension at most one, and show that this latter condition is also sufficient provided that the multiplicative subset $S$ consists of regular elements.

%

\begin{lem} \label{L:F_1-perp}
 Assume that $R$ is $S$\+almost perfect.
 Then, for every $n\geq1$, $\Ext^n_R(M, D)=0$ for every $M\in\F_1$
and every $S$\+divisible $R$\+module $D$.
\end{lem}
\begin{proof}
 Let $M\in\F_1$ and let $Y$ be an $R_S$\+module.
 Then, for every $n\geq 1$,
$\Ext^n_R(M, Y)\cong\Ext^n_{R_S}(M\otimes R_S, Y)$ and the last term is
zero since $R_S$ is a perfect ring.
 By Lemma~\ref{L:perfect-Matlis}, $\operatorname{p.dim}R_S\leq 1$.
 Let now $D$ be an $S$\+divisible $R$\+module.
 By Corollary~\ref{S-h-nil-S-h-divisible}, there is a short exact sequence:
\[
 0\la \Hom_R(R_S/\phi(R),D)\la \Hom_R(R_S,D)\la D\la 0,
\]
where $\Hom_R(R_S,D)$ is an $R_S$\+module and $\Hom_R(R_S/\phi(R),D)$
is $S$\+reduced and $S$\+weakly cotorsion, by
Lemma~\ref{L:torsion-Hom}~(1).
 Then, for every $n\geq 1$, we get the exact sequence:
\begin{multline*}
 0=\Ext^n_R(M,\Hom_R(R_S,D))\la \Ext^n_R(M,D) \\
 \la\Ext^{n+1}_R(M,\Hom_R(R_S/\phi(R),D))\la\dotsb
\end{multline*}
and the last Ext term is zero, since $M\in\F_1$ and,
by Theorem~\ref{T:characterization}~(3),
$\Hom_R(R_S/\phi(R),D)$ is Enochs cotorsion. 
Hence we conclude.
\end{proof}

We now can show that $\clP_1=\F_1$ whenever $R$ is an $S$\+almost perfect ring.

\begin{prop}\label{P:P_1=F_1}
 Assume that $R$ is an $S$-almost perfect ring.
 Then $\clP_1=\F_1$ and consequently $\clP_n=\F_n$
for every $n\geq 1$.
\end{prop}

\begin{proof}
 $\F_1$ is equal to $\clP_1$
if and only if $\F_1^{\perp_1} $ contains the cosyzygy modules, i.e.,
the epimorphic images of injective modules.
  Let $M\in \F_1$. By Lemma~\ref{L:F_1-perp}, $\Ext^n_R(M, D)=0$
for every $n\geq 1$ and every $S$\+divisible module $D$.
 Let now $C$ be a cosyzygy module and let $d(C)$ be its maximal
$S$\+divisible submodule.
 From the exact sequence $0\to d(C)\to C\to C/d(C)\to0 $ and from
the above remark, we see that it is enough to show that
$\Ext^1_R(M, C)=0$ for every $S$\+reduced cosyzygy module~$C$.

  Let $0\to F\to R^{(\alpha)}\to M\to 0$ be a presentation of $M$
with $F$ a flat module.
 By Theorem~\ref{T:characterization} and Lemma~\ref{L:SF},
there is a flat module $F_1$ such that $G=F\oplus F_1$ fits in
an exact sequence  of the form 
$0\to R^{(\beta)}\to G\overset{\pi}\to R_S^{(\gamma)}\to 0$
for some cardinals $\beta$, $\gamma$
and thus for $M$ we can consider the flat presentation
$0\to G\to R^{(\alpha)}\oplus F_1\to M\to 0$.
 Form the pushout diagram
\[\begin{CD}
@. 0@. 0\\
@. @VVV @VVV\\
@. R^{(\beta)}@= R^{(\beta)}\\
@. @VVV @VVV\\
0@>>>G@>>> R^{(\alpha)}\oplus F_1@>>>M@>>> 0\\
@. @VVV @VVV@| \\
0 @>>> R_S^{(\gamma)}@>>> P @>>> M @>>> 0 \\
@. @VVV @VVV\\
@.  0 @. 0
\end{CD}\]
 Then $P$ has projective dimension at most one, since
$\operatorname{p.dim}_RF_1\le1$ (as
$\operatorname{p.dim}_RR_S\le1$ implies that $S$\+$\SF\subset\clP_1$).
 Hence $\Ext^1_R(P, C)=0$ for every cosyzygy module $C$, and if
$C$ is moreover $S$\+reduced, then $\Hom_R(R_S^{(\gamma)}, C)=0$,
showing that $\Ext^1_R(M, C)=0$.
\end{proof}

\begin{lem}\label{L:F_1-R/I}
 For every $s\in S$, let\/ $\Ann(s)\subset R$ be the annihilator of~$s$
in $R$ and $M[s]\subset M$ be the annihilator of~$s$ in a module~$M$. 
 As above, we denote by $I$ the kernel of the localization
map $R\to R_S$.
 Then, for every $R$\+module $M$, the following hold true:
\begin{enumerate}
\item $\Tor^R_1(R/sR, M)=\frac{M[s]}{\Ann (s)M}$.
\item If $M\in \F_1$, then $\Tor^R_m(R/I, M)=0$ for every $m\geq 1$.
In particular, if $M\in\clP_1$, then $M/IM\in \clP_1(R/I).$
\end{enumerate}
\end{lem}

\begin{proof}
(1) Consider the short exact sequence
\begin{equation}
 \tag{$\circ$} 0\la R/\Ann (s)\overset{\dot{s}}\la R\la R/sR\la 0,
\end{equation}
where $\dot{s}$ denotes the multiplication by~$s$.
 The conclusion follows by considering the exact sequence:
 \[0\la\Tor^R_1(R/sR, M)\la M/\Ann (s)M\overset{\dot{s}}\la M.\]
(2) From the exact sequence~($\circ$) we infer that
$\Tor^R_2(R/sR, M)\cong \Tor^R_1( R/\Ann (s), M)$. 
 Let $M\in\F_1$. Then $\Tor^R_m(R/I, M)=0$ for every $m\geq 1$,
since $\Tor$ commutes with direct limits. 
 Hence $\Ext^n_R(M, X)\cong \Ext_{R/I}^n(M/IM, X)$ for every
$n\geq 1$ and every $R/I$\+module $X$,
from which the last statement is immediate.
\end{proof}

\begin{lem}\label{L:Enochs-cot}
 Assume that $R_S$ is a perfect ring and that $\F_0\subseteq \clP_1$.
 Then, every $S$\+divisible $R$\+module is Enochs cotorsion.
\end{lem}

\begin{proof}
 $R_S$ is a flat $R$-module, hence our assumption implies that
$\operatorname{p.dim}_RR_S\allowbreak\leq 1$.
 By Proposition~\ref{P:Matlis}, every $S$\+divisible module $D$ is
$S$\+h-divisible, and by Lemma~\ref{L:annihil}~(2)
$D\in \clP_1^{\perp_1}$.
 Since $\F_0\subseteq \clP_1$, we conclude that $D$ is Enochs
cotorsion.
\end{proof}

 We now consider the case when $S$ is a multiplicative
set of regular elements of~$R$.

\begin{prop}\label{P:regular}
 Assume that $S$ consists of regular elements.
 If $R_S$ is a perfect ring and $\clP_1=\F_1$, then $\clP_1^{\perp_1}$
coincides with the class of $S$\+divisible $R$\+modules.
\end{prop}

\begin{proof} 
 If $\clP_1=\F_1$, then $\operatorname{p.dim}_RR_S\leq 1$, and,
by Proposition~\ref{P:Matlis} and Lemma~\ref{L:annihil}~(2),
the class $\clP_1^{\perp_1}$ (\,=~$\F_1^{\perp_1}$) contains the class
of $S$\+divisible modules.
 Furthermore, $\F_1^{\perp_1}$ is contained in the class of
$S$\+divisible modules, since $R_S/R\in\F_1$ and
Lemma~\ref{L:annihil}~(1) applies.
\end{proof}

\begin{lem}\label{L:P_1=F_1}
 Assume that $\clP_1=\F_1$.
 Then, for every regular element $r\in R$, the ring $R/rR$ is perfect.
\end{lem}

\begin{proof}
 It is enough to show that every Bass $R/rR$\+module is projective.
 Let $\overline a_1$ , $\overline a_2$,~\dots, $\overline a_n$, ~\dots\
be a sequence of elements in $R/rR$.
 Consider the direct system $R/rR\overset{\overline a_1}\to R/rR
\overset{\overline a_2}\to R/rR\overset{\overline a_3}\to\dots$
and its direct limit~$N$.
 Let
\begin{equation}
 \tag{$\ast$} \textstyle
 0\la \bigoplus_n R/rR \overset{\psi}\la \bigoplus_n R/rR \la N \la 0
\end{equation}
be a direct limit presentation of~$N$.
 The sequence $(\ast)$ is pure also as a sequence of $R$\+modules
(since for every $R/rR$\+module $M$ and every $R$\+module $X$,
\,$M\otimes_RX\cong M\otimes_{R/rR}X/rX$, hence the functor
${-}\otimes_RX$ leaves the sequence exact).
 Thus $\Tor^R_2(N,X)=0$ for every $R$-module $X$, since
$R/rR\in \clP_1\subseteq \F_1$ and $\Tor^R_1(\psi,X)$ is a monomorphism.
 So $N\in \F_1$ and by assumption $N\in \clP_1$.
  From sequence~($\ast$) we infer that the projective dimension of $N$
over $R/rR$ is at most~$1$, and since $r$~is a regular element of $R$
the Change of Rings Theorem tells us that
$\operatorname{p.dim}_R N=\operatorname{p.dim}_{R/rR}N+1$.
 Hence $\operatorname{p.dim}_{R/rR}N=0$.
\end{proof} 

 If the multiplicative set $S$ consists of regular elements, we can state
a result analogous to \cite[Theorem~7.1]{FS}, which was formulated for
the case when $S$ is the set of all the regular elements of~$R$.
 In particular, an equivalent condition is given by $\clP_1=\F_1$.

\begin{prop}\label{P:regular-P_1}
 Let $S$ be a multiplicative set of regular elements of a ring $R$ and
assume that $R_S$ is a perfect ring.
 Then the following are equivalent:
\begin{enumerate}
\item $R/sR$ is a perfect ring for every $s\in S$.
\item $\F_1^{\perp_1}$ coincides with the class of $S$\+divisible
        $R$\+modules.
\item $\F_1^{\perp_1}$ is closed under epimorphic images.
\item $\clP_1=\F_1$.
\item If $\D_S$ is the class of $S$\+divisible $R$\+modules, then
        $(\clP_1, \D_S)$ is a cotorsion pair which coincides with
        $(\F_1, \F_1^{\perp_1})$. 
\end{enumerate}
\end{prop}

\begin{proof}

(1) $\Rightarrow$ (2)  By Lemma~\ref{L:F_1-perp} every $S$\+divisible
module is contained in $\F_1^\perp$.

 Conversely, $\F_1^{\perp_1}\subseteq \{R_S/R\}^{\perp_1}$ and
$\{R_S/R\}^{\perp_1}$ is the class of $S$\+divisible modules by
Lemma~\ref{L:annihil}~(1).

(2) $\Rightarrow$ (3) is obvious.

(3) $\Rightarrow$ (4) Let $M\in \F_1$ and for every $R$\+module $X$
consider a short exact sequence $0\to X\to E\to E/X\to 0$ with $E$
an injective module.
 Then $\Ext^1_R(M, E/X)\cong\Ext^2_R(M, X)$ and by~(3)
\,$\Ext^1_R(M, E/X)=0$.

 Thus $M\in \clP_1$.

(4) $\Rightarrow$ (5) Follows by Proposition~\ref{P:regular}.

(5) $\Rightarrow$ (1) Follows by Lemma~\ref{L:P_1=F_1}.
\end{proof}

 In the above proposition, the assumption that $R_S$ be a perfect ring
cannot be dropped.
 In fact, the example below shows there exists a commutative ring
such that $\clP_1=\F_1$, but with nonperfect total ring of quotients.

\begin{expl}
 In~\cite[5.1]{Ber} it is shown that there is a totally disconnected
topological space $X$ whose ring of continuous functions $K$ is
Von Neuman regular and hereditary.
 Moreover, every regular element of $K$ is invertible.
 Hence $K$ coincides with its own ring of quotients and
$\clP_1=\F_0=\Modr K$, but $K$ is not perfect,
since it is not semisimple.
\end{expl}

\end{document}